\documentclass{article}
\usepackage{amsmath,amsfonts,amssymb,amsthm,graphicx}
 \usepackage{dsfont}    
\usepackage{enumerate}
\usepackage{enumitem}
\usepackage[numbers]{natbib}
\usepackage{multirow}
\usepackage{comment}
\usepackage[margin=1.1in]{geometry}

\usepackage{booktabs,subcaption,dcolumn}

\usepackage{algorithm}
\usepackage[noend]{algpseudocode}  
\usepackage[colorlinks=true,citecolor=blue,pdfpagemode=UseNone,pdfstartview=FitH]{hyperref}

\usepackage{graphicx} 
\usepackage{tikz}

\usetikzlibrary{graphs,decorations.pathmorphing,decorations.markings} 
\usetikzlibrary{arrows}

\newif\ifFULL
\ifFULL

\fi

\newif\ifRuodu
\ifRuodu
  \renewcommand{\ge}{\geqslant}
  \renewcommand{\le}{\leqslant}
  \renewcommand{\epsilon}{\varepsilon}
  
\fi

\renewcommand{\d}{\,\mathrm{d}}

\newcommand{\p}{\mathbb{P}}

\newcommand{\E}{\mathbb{E}}    
\newcommand{\R}{\mathbb{R}}    

  \newcommand{\id}{\mathds{1}}

\newcommand{\cD}{\mathcal{D}}

\usepackage{setspace}

\theoremstyle{plain}
\newtheorem{theorem}{Theorem} 
\newtheorem{corollary}[theorem]{Corollary}
\newtheorem{lemma}[theorem]{Lemma}
\newtheorem{proposition}[theorem]{Proposition}

\theoremstyle{definition}

\newtheorem{example}[theorem]{Example}

\theoremstyle{remark}
\newtheorem{remark}[theorem]{Remark}

 \renewcommand{\cite}{\citet}  

\title{Multiple testing under negative dependence} 
\author{Ziyu Chi\thanks%
  {Dept. of IEOR,
University of California Berkeley,
  Berkeley CA, USA.
  \href{ziyu.chi@berkeley.edu}{ziyu.chi@berkeley.edu}.}
  \and
  Aaditya Ramdas\thanks
  {
Depts. of Statistics and Machine Learning, Carnegie Mellon University, Pittsburgh PA, USA.
  \href{aramdas@cmu.edu}{aramdas@cmu.edu}.
  }
  \and
  Ruodu Wang\thanks%
  {Dept. of Statistics and Actuarial Science,
 University of Waterloo,
  Waterloo ON, Canada.
  \href{mailto:wang@uwaterloo.ca}{wang@uwaterloo.ca}.}}

\begin{document}

\maketitle

\begin{abstract}
The multiple testing literature has primarily dealt with three types of dependence assumptions between p-values: independence, positive regression dependence, and arbitrary dependence. In this paper, we provide what we believe are the first theoretical results under various notions of negative dependence (negative Gaussian dependence, negative regression dependence, negative association, negative orthant dependence and weak negative dependence). These include the Simes global null test and the Benjamini-Hochberg procedure, which are known experimentally to be anti-conservative under negative dependence. The anti-conservativeness of these procedures is bounded by factors smaller than that under arbitrary dependence (in particular, by factors independent of the number of hypotheses). We also provide new results about negatively dependent e-values, and provide several examples as to when negative dependence may arise. Our proofs are elementary and short, thus amenable to extensions. 
\end{abstract}

\tableofcontents

\section{Introduction}

Ever since the seminal book by~\cite{tukey1953problem}, the subfield of multiple comparisons and multiple hypothesis testing has grown rapidly and found innumerable applications in the sciences.  However, 
some relatively basic theoretical questions remain unsolved. For example, we have not encountered concrete theoretical results on the performance of the Benjamini-Hochberg (BH) procedure~\citep{BH95} when the p-values are negatively dependent. Closely related to the BH procedure is the Simes global null test, for which we have also not seen results under negative dependence. This paper begins to fill the aforementioned gaps, and paves the way for more progress in this area.

One of the main challenges is that there are many definitions of what it means to be ``negatively dependent''.
In the Gaussian setting, the definitions simply amount to the signs of covariances being positive or negative, but one often cares about more nonparametric definitions that apply more generally, and these are aplenty. It is apriori unclear which definition of dependence will  (A) lend itself to analytical tractability for bounding error rates of procedures, (B) have enough examples satisfying the definition so as to potentially yield practical insights in some situations. Once a suitable definition has been adopted, further choices must be made: one must specify whether the dependence is being assumed across all p-values or only those that are null (for example). 
The multitude of possibilities is daunting, perhaps explaining the lack of progress. 

The above combination of (A) and (B) has been arguably successfully achieved for positive dependence. \cite{S98} published an important result settling the Simes conjecture under a notion of positive dependence called multivariate total positivity of order two, that was studied in depth by~\cite{karlin1980classes} in 1980. \cite{BY01} strengthened and extended Sarkar's result: they showed that the BH procedure controls the false discovery rate (FDR) under a weaker condition called positive regression dependence on a subset (PRDS). This notion too goes back several decades to~\cite{lehmann1966some}, who proposed PRD in a bivariate context, and (the elder)~\cite{sarkar1969some}, who generalized PRD to a multivariate context. This paper will provide the first results under the negative dependence analog of the PRD condition.

The aforementioned 2001 paper also proved that under arbitrary dependence, the BH procedure run at target level $\alpha$ on $K$ hypotheses could have its achieved FDR control be inflated a factor of about $\log K$ (sometimes called the Benjamini-Yekutieli or BY correction). This is a huge inflation in modern contexts where $K$ can be in the millions or more. The above results have arguably led to a practical dillema.
When the BH procedure is applied in situations where PRDS is a questionable assumption (or is in fact known to not hold), should one apply the aforementioned BY correction? Theoretically perhaps one should use the correction, but we have rarely seen the BY correction used in practice because it hurts power a lot. 

While we understand the practice of not using the BY correction, the  gap between theory and practice is mildly unsettling. One way out is to seek a better theoretical understanding of what types of assumptions result in inflation factors of much less than $\log K$, along with some justification that these could occur in practice (points (A) and (B) from earlier). 

It is in the above context that we see that the current paper makes some novel and arguably important contributions to the literature. Of course, by virtue of being the first, as far as we are aware, nontrivial result on the performance of Simes and BH under negative dependence, it will hopefully stimulate future progress. (In fact, we are only aware of one other recent mutliple testing work by~\citet{gou2018hochberg} for a different procedure (the Hochberg method) under negative dependence.) But equally importantly, the bounds are derived under a very weak notion of negative dependence (and thus easier to satisfy), and the error inflation factors (or anticonservativeness) are proved to be independent of the number of hypotheses $K$, only involving explicit and small constant factors. Thus, the result is not overly pessimistic, and is a stepping stone to bridging theoretical progress with practical advice. 

 The rest of this paper is organized as follows. Section~\ref{sec:defns-negdep} presents a few key notions of negative dependence, along with some examples of when they occur. Section~\ref{sec:merge-p} presents results on the Simes test using negatively dependent p-values. Section~\ref{sec:merge-e} briefly discusses the case of negatively dependent e-values. Section~\ref{sec:fdr} builds on Section~\ref{sec:merge-p} to derive results on the FDR of the BH procedure under negative dependence. Section~\ref{sec:sims} presents simulation results, before we conclude in Section~\ref{sec:conc}.

\section{Notions of negative dependence}\label{sec:defns-negdep}

Fix an atomless probability space $(\Omega,\mathcal F,\p)$ where all random variables live.   
The aim of this section is to introduce several important notions of negative dependence, summarizing some properties and referencing proofs for the following implications along the way in Figure~\ref{fig:notionsofnegdep}.

\begin{figure}[ht]
    \centering

\begin{tikzpicture} [>=implies]
\pgfsetblendmode{multiply}
    \node[minimum size=8mm] (PCT) at (0,0) {Counter-monotonicity~\eqref{eq:def-counter}};    \node[minimum size=8mm] (PCT1) at (1.5,0) {~~~~~~~~~~};
    \node[minimum size=8mm] (NA) at (0, -1.5) {Negative association~\eqref{eq:def-na}};  
    \node[minimum size=8mm] (NA1) at (1.3, -1.5) {~~~~~~~~~~};
    \node[minimum size=8mm] (NGD) at (6, 0) {Negative Gaussian dependence~\eqref{eq:def-neg-gau}};
        \node[minimum size=8mm] (NGD1) at (3.8, 0) {~~~~~~~~~~};
    \node[minimum size=8mm] (NDTSO) at (6, -1.5){~~~~~~~~~~~~~~~~~~~~~~~~~};
        \node[minimum size=8mm] (NDTSO1) at (4, -1.5){~~~~~~~~~~};
    \node[minimum size=8mm] (NDTSO-sh) at (6, -1.5) {Negative regression dependence~\eqref{eq:nds}};
    \node[minimum size=8mm] (NOD) at (3, -3) {Negative orthant dependence~\eqref{eq:def-nlod}~\mbox{plus}~\eqref{eq:def-nuod}};
       \node[minimum size=8mm] (NOD1) at (1.5, -3) {~~~~~};
       \node[minimum size=8mm] (NOD2) at (4.5, -3) {~~~~~};
    \node[minimum size=8mm] (NLOD) at (0, -4.5) {Negative lower orthant dependence~\eqref{eq:def-nlod}};
    \node[minimum size=8mm] (NUOD) at (7, -4.5) {Negative upper orthant dependence~\eqref{eq:def-nuod}};
    \node[minimum size=8mm] (LWND) at (0, -6) {(Lower) weak negative dependence~\eqref{eq:def-neg}};
    \node[minimum size=8mm] (UWND) at (7, -6) {(Upper) weak negative dependence~\eqref{eq:def-neg-2}};
        \node[minimum size=8mm] (NLOD1) at (1, -4.5) {~~~~};
    \node[minimum size=8mm] (NUOD1) at (5, -4.5) {~~};
    
    \begin{scope}[transparency group]
        \draw[->, double, double distance=1mm]  (PCT)--(NA);
    \end{scope}
    \begin{scope}[transparency group]
        \draw[->, double, double distance=1mm]  (NGD)--(NDTSO);
    \end{scope}
    \begin{scope}[transparency group]
        \draw[->, double, double distance=1mm]  (PCT1)--(NDTSO1);
    \end{scope}
    \begin{scope}[transparency group]
        \draw[->, double, double distance=1mm]  (NGD1)--(NA1);
    \end{scope}  
    \begin{scope}[transparency group]
        \draw[->, double, double distance=1mm]  (NA)--(NOD1);
    \end{scope}
    \begin{scope}[transparency group]
        \draw[->, double, double distance=1mm]  (NDTSO)--(NOD2);
    \end{scope}
    \begin{scope}[transparency group]
        \draw[->, double, double distance=1mm]  (NOD1)--(NLOD1);
    \end{scope}
    \begin{scope}[transparency group]
        \draw[->, double, double distance=1mm]  (NOD2)--(NUOD1);
    \end{scope}
    \begin{scope}[transparency group]
        \draw[->, double, double distance=1mm]  (NLOD)--(LWND);
    \end{scope}
    \begin{scope}[transparency group]
        \draw[->, double, double distance=1mm]  (NUOD)--(UWND);
    \end{scope}
\end{tikzpicture}
\caption{Notions of negative dependence.}
\label{fig:notionsofnegdep}
\end{figure}
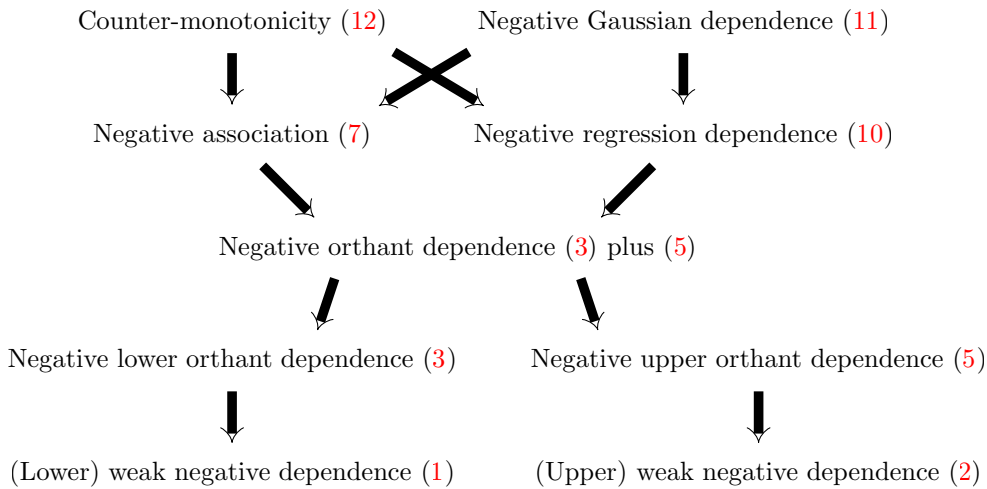

  We define all these notions below, from the weakest to  the strongest. 
  For a  random vector
$\mathbf X=(X_1,\dots,X_K)$,  
  let $F_k$ be the distribution function of $X_k$ for $k\in \mathcal K$.
To begin, we say that 
$\mathbf X$ is  (lower) \emph{weakly negatively dependent} if   
\begin{align}\label{eq:def-neg}
\p\left(\bigcap_{k\in A} \{X_k\le F_k^{-1}(p)\}\right)\le \prod_{k\in A} \p\left(  X_k\le F_k^{-1}(p)\right)  \mbox{~~~for all $A\subseteq \mathcal K$ and $p\in (0,1)$}.
\end{align} 
We will sometimes write ``$X_1,\dots,X_K$ are weakly negatively dependent'' instead of ``$\mathbf X$ is weakly negatively dependent'' (also for other notions of dependence), and this should cause no confusion.
Upper weak negative dependence can be defined by \begin{align}\label{eq:def-neg-2}
\p\left(\bigcap_{k\in A} \{X_k >  F_k^{-1}(p)\}\right)\le \prod_{k\in A} \p\left(  X_k> F_k^{-1}(p) \right)  \mbox{~~~for all $A\subseteq \mathcal K$ and $p\in (0,1)$},
\end{align} but we will only need the lower version \eqref{eq:def-neg}, and so we omit the qualifier ``lower'' going forward.

Condition \eqref{eq:def-neg} is weaker than the notion of  \emph{negative lower orthant dependence} of \cite{BSS82}, which is defined by 
 \begin{align}\label{eq:def-nlod}
\p\left(\bigcap_{k\in \mathcal K } \{X_k\le x_k\}\right)\le \prod_{k\in  \mathcal K  } \p\left(  X_k\le x_k\right) \mbox{~~~for all $(x_1,\dots,x_K) \in \R^K$}.
 \end{align}
Indeed, we can see that \eqref{eq:def-nlod} implies~\eqref{eq:def-neg} by taking $x_k\to \infty$ for $k \notin A$ and $x_k=x$ for $k \in A$. 
Further, $\mathbf X$ is negative lower orthant dependent if and only if \begin{align}
\label{eq:nlod-f}
\E\left[\prod_{k=1}^K \phi_k(X_k)\right] \leq \prod_{k=1}^K \E\left [ \phi_k(X_k)\right ] \mbox{~~for all nonnegative decreasing functions $\phi_1,\dots\phi_K$};
\end{align} see
 Theorem
6.G.1 (b) of \cite{SS07} or Theorem 3.3.16 of \cite{MS02}. All terms like ``increasing'' and ``decreasing'' are in the non-strict sense.

There is a related notion of \emph{negative upper orthant dependence}:
  \begin{align}\label{eq:def-nuod}
\p\left(\bigcap_{k\in \mathcal K } \{X_k> x_k\}\right)\le \prod_{k\in  \mathcal K  } \p\left(  X_k >  x_k\right) \mbox{~~~for all $(x_1,\dots,x_K) \in \R^K$}.
 \end{align}
 Similarly to \eqref{eq:nlod-f}, negative upper orthant dependence is equivalent to  
 \begin{align}
\label{eq:nuod-f}
\E\left[\prod_{k=1}^K \phi_k(X_k)\right] \leq \prod_{k=1}^K \E\left [ \phi_k(X_k)\right ] \mbox{~~for all nonnegative increasing functions $\phi_1,\dots\phi_K$.}
\end{align}

\emph{Negative orthant dependence} means that both negative lower orthant dependence and negative upper orthant dependence hold simultaneously.

Negative orthant dependence is in turn weaker than \emph{negative association} of $\mathbf X$, which requires that for any disjoint subsets $A,B \subseteq \mathcal K$, and any real-valued, coordinatewise increasing functions $f,g$, 
\begin{align}\label{eq:def-na}
\text{Cov}(f(\mathbf X_A),g(\mathbf X_B)) \leq 0,
\end{align}
where $\mathbf X_A=(X_k)_{k\in A}$, $\mathbf X_B=(X_k)_{k\in B}$,
if $f(\mathbf X_A),g(\mathbf X_B)$ have finite second moments. Equivalently, 
\begin{align}\label{eq:def-na2}
\E[f(\mathbf X_A)g(\mathbf X_B)] \leq \E[f(\mathbf 
 X_A)]\E[g(\mathbf  X_B)].
\end{align}
This in turn implies that for any non-overlapping sets $\{A_k\}_{k=1,\dots,\ell}$ and  nonnegative increasing functions $\{\phi_k\}_{k=1,\dots,\ell}$, we have
\begin{equation}\label{eq:na-implication}
\E\left[\prod_{k=1}^\ell \phi_k(\mathbf 
 X_{A_k})\right] \leq 
\prod_{k=1}^\ell\E[\phi_k(\mathbf 
 X_{A_k})].
\end{equation}
It is necessary and sufficient to require $f$ and $g$ in \eqref{eq:def-na} to be bounded, which can be seen from an  approximation argument.  For negatively associated random variables, all pairwise correlations are non-positive. Thus, 
\(
\mathrm{Var}(\sum_{k=1}^K X_k) \leq \sum_{k=1}^K \mathrm{Var}(X_k).
\)
\cite{shao2000comparison} proved the following coupling result. Let $X_1,\dots,X_K$ be negatively associated, and let $X^*_1,\dots,X^*_K$ be independent random variables such that $X_k$ and $X^*_k$ have the same (marginal) distribution for each $k$. 
Then, for all convex functions $f$,
\(
\mathbb E[ f(\sum_{k=1}^K X_k)] \leq \mathbb E[{f(\sum_{k=1}^K X^*_k)}].
\)
 
A random vector $\mathbf X$ is said to be stochastically decreasing in $Y$ if $\mathbb E[g(\mathbf X)\mid Y=y]$ is decreasing in $y$ whenever $g$ is a coordinatewise increasing function such that the conditional expectation exists. $\mathbf X$ is \emph{negative regression dependent} if  
\begin{align}\label{eq:nds}
\text{$\mathbf X^{-i}$ is stochastically decreasing in $X_i$ for every $i$, }
\end{align}
where $\mathbf X^{-i}$ is the vector formed by deleting the $i$-th coordinate of $\mathbf X$. 
The notion in  \eqref{eq:nds} is called negative dependence through stochastic ordering by \cite{block1985concept}, and it is the negative analog of the famous positive regression dependence condition (also called positive  dependence through stochastic ordering) frequently encountered in multiple testing, under which the Simes test and the Benjamini-Hochberg procedure (both defined later in the paper) are known to be conservative. 
The term negative regression dependence was used by \cite{lehmann1966some} in case $K=2$.

A random vector $\mathbf X $ is  \emph{Gaussian dependent} if  there exist increasing functions (or decreasing functions)  $f_1,\dots,f_K$ and
a Gaussian vector $(Y_1,\dots,Y_K)$ such that 
$ X_k=f_k(Y_k)$ for $k\in \mathcal K$.
The correlation matrix of $(Y_1,\dots,Y_K)$ is called a \emph{Gaussian correlation} of $\mathbf X$, which is unique if $\mathbf X$ has continuous marginals.
For instance, if $Y_1,\dots,Y_K$  are standard Gaussian test statistics and $P_1,\dots,P_K$ are the produced one-sided p-values (as $P_k = \Phi(-Y_k)$, where $\Phi$ is the standard Gaussian CDF), then $\mathbf P$ is Gaussian dependent.
 Further, $\mathbf X$ is  \emph{negatively Gaussian dependent} 
    if it is Gaussian dependent
    and its Gaussian correlation coefficients are non-positive ($ X_k=f_k(Y_k)$ for $k\in \mathcal K$):
    \begin{align}\label{eq:def-neg-gau}
    (Y_1,\dots,Y_K) \sim \mathrm N(\mu, \Sigma), \text{ for some $\mu, \Sigma$ such that } \Sigma_{ij} \leq 0 \text{ for all } i\neq j.
\end{align}


If  $\mathbf X$ is negatively Gaussian dependent, then $\mathbf X$ has both negative association  and  negative regression dependence, implying negative lower orthant  dependence and weak negative dependence; see \citet[Section 3.4]{JP83} and also Lemma~\ref{lem:gaussian} in Section~\ref{sec:gaussian}. The   statement on negative orthant dependence can be verified directly by Slepian's lemma~\citep{S62}. 

Finally, there exists an ``extremal (most) negative dependence'': 
$(X,Y)$ is \emph{counter-monotonic} if there exists increasing functions $f,g$ and a random variable $Z$ such that 
 $
(X,Y)=(f(Z),-g(Z))$ almost surely. This can be alternatively stated as 
$(X(\omega)-X(\omega'))(Y(\omega)-Y(\omega'))\le 0$ for almost every pair of $(\omega,\omega')\in \Omega^2.$
In higher dimensions, a random vector $\mathbf X$ is counter-monotonic if each pair of its components is counter-monotonic, that is,
\begin{equation}\label{eq:def-counter}
(X_i(\omega)-X_i(\omega'))(X_j(\omega)-X_j(\omega'))\le 0 \text{ for almost every   $(\omega,\omega')\in \Omega^2$ and every distinct $i,j$.}
\end{equation}
The structure \eqref{eq:def-counter} imposes strong conditions on the marginal distributions (in particular, the marginal distributions cannot be continuous when $K\ge 3$).
A pairwise counter-monotonic random vector has the smallest joint distribution function among all random vectors with the same marginals. See \cite{PW15} for the above statements and other forms of extremal negative dependence.
As shown by \cite{LLW23}, a pairwise counter-monotonic random vector is 
negatively associated 
and negative regression dependent.  

\smallskip

\noindent \textbf{Closure properties.} 
We mention a few relevant closure properties below.

\textit{Monotone transformations:}
All notions of negative dependence 
are preserved under concordant coordinatewise monotonic transformations (the term \emph{concordant} here means that we apply either  decreasing transformations to all coordinates or increasing transformations to all coordinates).  

\textit{Convolution:} Suppose that  $\mathbf X_1$ and $\mathbf X_2$ are independent of each other. 
If each of $\mathbf X_1$ and $\mathbf X_2$ is negative lower orthant dependent, then $\mathbf X_1 + \mathbf X_2$ is also negative lower orthant dependent; see \citep[Theorem 4.2]{M97} and \citep[Corollary 3]{MR22}. 
This also holds for negative association by combining Properties P6 and P7 of \cite{JP83}.

\textit{Concatenation:} If $\mathbf X_1$ and $\mathbf X_2$ are each negatively associated (or negative regression dependent), and are independent of each other, then so is their concatenation $(\mathbf X_1, \mathbf X_2)$.
An example of this type is given by $(X_1,-X_1,X_2,-X_2,\dots,X_K,-X_K)$ for any independent $X_1,\dots,X_K$. 


\textit{Marginalization:} If $\mathbf X$  satisfies any notion of negative dependence mentioned in this section, then so $\mathbf X_A$ for any  nonempty $A\subseteq \mathcal K$. This can be verified directly from the definition of these concepts. 

\subsection*{Examples of negative dependence}
Beyond the Gaussian case mentioned above, some other simple examples may be useful for the reader to keep in mind going forward. These examples can be found in e.g., \cite{JP83}.

\textit{Categorical distribution:} Suppose that $\mathbf X$ is a draw from a categorical distribution with $K$ categories, meaning that it is a binary vector that sums to one. Then $\mathbf X$ is counter-monotonic, thus both negatively associated and negative regression dependent.   

\textit{Multinomial distribution ($m$ balls in $K$ bins):} If $\mathbf X$ is a draw from a multinoulli distribution, meaning that it takes values in $\{0,1\}^K$ and  $\sum_{i=1}^K X_i = m$, with each of the $K \choose m$ possibilities being equally likely, then $\mathbf X$ is negatively associated as well as negative regression dependent. If the possibilities are not all equally likely, $\mathbf X$ is negatively orthant dependent~\citep{BSS82}. The same conclusions hold for a draw from a multinomial distribution, since these are just sums of mutually independent multinoullis.

\textit{Uniform permutations:} Let $\mathbf X$ be a uniformly random permutation of some fixed vector $(x_1,\dots,x_n)$. Then $\mathbf X$ is negatively associated. 

\textit{Sampling without replacement:} Along similar lines to the above, sampling without replacement leads to negatively associated random variables. To elaborate: suppose $X_1,\dots,X_K$ are sampled without replacement from a bag containing $N\ge K$ numbers. Then $\mathbf X$ is negatively associated. 

\textit{Recentered Gaussians:} \cite{block1985concept} pointed out that if $X_1,\dots,X_K$ are iid Gaussians, and $\bar X:= (X_1+\dots+X_K)/K$, then $\mathbf Z:=(X_1-\bar X,\dots,X_K - \bar X)$ is multivariate Gaussian with negative correlations, and thus negative regression dependent and negatively associated, for example. 
Hence, so are the p-values obtained  by  $P_k=\Phi(Z_k/(1-1/K)^{1/2} ) $ for $k\in \mathcal K$. 
\cite{block1985concept} also show that several other examples, like multivariate negative binomial, Dirichlet and multivariate hypergeometric, are all negative regression dependent.

\textit{Tournament performance scores:}  Data summarizing tournament performance are often negatively associated. We summarize some examples below. Consider a round-robin tournament between $K$ players, summarized by a pairwise game matrix $X$ of size $K \times K$. 
The first three examples below can be shown using Properties P6 and P6 of \cite{JP83}, by noting that the scores $X_{ij}$ and $X_{ji}$ are counter-monotonic, and they are independent of the other scores. 
Therefore, scores in the $K\times K$ matrix  are negatively associated \citep[P7]{JP83}, and so are their row sums \citep[P6]{JP83}.
The last example is shown by \cite{MR22}. 
 
 \textit{Binary outcomes.} Suppose each game ends in a win or loss. Let $p_{ij}$ denote the probability that $i$ beats $j$ and they play $n_{ij}$ games against each other. Assuming that all games are independent, we have $X_{ij}\sim \text{Binomial}(n_{ij},p_{ij})$. Let us calculate scores of the $K$ players as $S_i = \sum_{j \neq i} X_{ij}$, and denote $\mathbf{S} = (S_1,\dots,S_K)$. Then $\mathbf S$ is negatively associated. (This actually improves a not-so-well known result by~\cite{huber1963remark} who proved $\mathbf S$ is negative lower orthant dependent.)  

 \textit{Constant sum games.} Suppose at the end of their game(s), each pair of players split a reward $r_{ij} \geq 0$, meaning that the rewards $X_{ij},X_{ji}$ are nonnegative and sum to $r_{ij}$. Defining each player's scores as before, $S_i = \sum_{j \neq i} X_{ij}$, we have that $\mathbf{S}$ is negatively associated. (Obviously this example generalizes the previous one, and even allows for ties.) 

 \textit{Random-sum games.} If the aforementioned rewards $r_{ij}$ are themselves random variables, $\mathbf{S}$ remains negatively associated, as long as $(X_{ij},X_{ji})$ is counter-monotonic. This happens in soccer, where the winning team is often awarded three points (and the losing team zero), but if the match is drawn, both teams get one point. This means that $(X_{ij},X_{ji})$ can take the value $(3,0)$, $(1,1)$ or $(0,3)$.

 \textit{Knockout tournaments.} Moving beyond round-robin tournaments to knockout tournaments like in tennis grand slams, let $S_i$ denote the total number of games won by player $i$. For example, with 64 players, only the winner $i$ will have $S_i=6$, the runner-up $j$ will have $S_j=5$, the semifinal losers $k$ will have $S_k=4$, and those that lost in the first round have $S_\ell=0$. Suppose further that all players are of equal strength, meaning that all outcomes are fair coin flips. For a completely random schedule of matches,  $\mathbf S$ is negatively associated. For nonrandom draws (such as via player seedings/rankings), $\mathbf S$ is negative orthant dependent.

\bigskip

Appendix A in the supplementary material contains hypothesis tests that give rise to negative dependence. The examples discussed there are tests based on split samples, testing the mean of a bag of numbers, round-robin tournaments, and cyclical or ordered comparisons.

\section{Merging negatively dependent p-values}\label{sec:merge-p}

We begin with a recap of some well known properties of the Simes global null test, before turning to the new results under negative dependence.

\subsection{Recap: merging p-values with the Simes function} 
 Throughout, $K$ is a positive integer,
 and $\mathbf P=(P_1,\dots,P_K)$  is a random vector taking values in $[0,1]^K$. 
 Let $\p$ be the true probability measure and 
write $\mathcal K=\{1,\dots,K\}$.
 Following \cite{VW20}, 
a \emph{p-variable} $P$ is a  random variable  that satisfies $\p(P\le \alpha)\le \alpha$ for all $\alpha \in (0,1)$. 
Let $\mathcal U$ be the set of all standard uniform random variables under $\p$.
  
 We first consider the setting of testing a global null. 
In this setting, we will always assume each of $P_1,\dots,P_K$ is uniformly distributed on $[0,1]$ (thus in $\mathcal U$), and this is without loss of generality. 
Slightly abusing the terminology, we also call $P_1,\dots,P_K$ p-values.

For $p_1,\dots,p_K\in [0,1]$ and  $k\in \mathcal K$, let $p_{(k)}$ be the $k$-th order statistics of $p_1,\dots,p_K$  from the smallest to the largest.  
   Let $S_K:[0,1]^K\to [0,1]$ be the \emph{Simes function}, defined as 
   $$
   S_K(p_1,\dots,p_K)= \bigwedge_{k=1}^K \frac{K}{k} p_{(k)},
   $$
   where $a\wedge b := \min(a,b)$.
  Applying $S_K$ to $\mathbf P$ and choosing a fixed threshold $\alpha\in (0,1)$, we obtain the \emph{Simes test}  by rejecting the global null if $S_K(\mathbf P)\le \alpha$. The type-1 error of this test is $\p(S_K(\mathbf P)\le \alpha)$. 
   
We begin from the observation that the Simes inequality  
\begin{align} \label{eq:simes}
\p( S_K(\mathbf P)\le\alpha)\le   \alpha 
\mbox{~~~ for all $\alpha \in (0,1)$}
\end{align}
holds for a wide class of dependence structures of $\mathbf P$. 
 It is shown by \cite{S86} that if p-values $P_1,\dots,P_K$ are independent or comonotonic (thus identical), then
 \begin{align}
\label{eq:simes2}
\p( S_K(\mathbf P)\le\alpha) = \alpha 
\mbox{ for all $\alpha \in (0,1)$,}
\end{align} 
and thus \eqref{eq:simes} holds as an equality.  
Moreover, the inequality \eqref{eq:simes} holds for more general dependence structures; see e.g., \cite{S98} and \cite{BY01}. {The inequality \eqref{eq:simes} may also hold in an asymptotic sense; see e.g., \cite{finner2017simes}.}
Let us define the notion of \emph{positive regression dependence} (PRD). 
A set $A\subseteq \R^K$ is said to be \emph{increasing}
if $\mathbf x\in A$ implies $\mathbf y\in A$ for all   $\mathbf y\ge \mathbf x$.   
  A random vector $\mathbf P$ of p-values is PRD if for any  $k\in \mathcal K$ and   increasing set $A  \subseteq  \R^K$, the
function $x\mapsto \p(\mathbf P\in A\mid P_k\le x)$ is increasing on $ [0,1]$.  
 
\begin{proposition}[\cite{BY01}]\label{prop:PRD}
If the  vector of p-values $\mathbf P$ is PRD, then \eqref{eq:simes} holds. 
\end{proposition}

If $\mathbf P$ is Gaussian dependent (i.e., obtained from jointly Gaussian statistics; see
 Section~\ref{sec:gaussian})
and its pair-wise correlations are non-negative, 
then $\mathbf P$ satisfies PRD. In this case,  \eqref{eq:simes} holds by Proposition~\ref{prop:PRD}.   
When the correlations are allowed to be negative, things are slightly different:
\cite{HR95} showed that, for $K=2$ and some Gaussian-dependent $\mathbf P$ with negative correlation, 
\begin{equation}\label{eq:simes-empirical}
  \p( S_2(\mathbf P)\le0.05)  \approx 0.0501.
\end{equation}
Thus, \eqref{eq:simes} is slightly violated. The maximum value of $\p( S_K(\mathbf P)\le\alpha)  $
over all possible dependence structures of $\mathbf P$ is known  (\cite{H83}) to be:
\begin{align}\label{eq:simes3}
\max_{\mathbf P \in \mathcal U^K} \p( S_K(\mathbf P)\le\alpha) &= (\ell_K \alpha)\wedge 1
\mbox{~~~ for  $\alpha \in (0,1)$,} \\
\text{where  } \ell_K &:=\sum_{k=1}^K \frac 1k \approx \log K. \label{eq:ellK}
 \end{align}
There are several other methods of merging p-values under arbitrary dependence \citep{VW20,VWW22}.
 Here, we focus on negatively dependent p-values and e-values.

\subsection{Simes under negative dependence}
 
We next give a nontrivial upper bound on $\p(S_K(\mathbf P)\le \alpha)$ when $\mathbf P$ is weakly negatively dependent.

\begin{theorem}[Additive error bounds]\label{th:neg}
For every weakly negatively dependent   $\mathbf P\in \mathcal U^K$, we have  
\begin{align}\label{eq:simes-general}
\p( S_K(\mathbf P)\le\alpha)\le 
\alpha + \sum_{k=2}^K\binom{K}{k} \left(\frac{\alpha k}{K}\right)^k.
\end{align}
We can obtain the more succinct bound that does not depend on $K$,
\begin{align} \label{eq:simes-neg}
\p( S_K(\mathbf P)\le\alpha)\le  
   \alpha + 2 \alpha^ 2   + \frac{9}{2}\alpha^3 +   \frac{1}{  \sqrt{8\pi}}\frac{(e\alpha)^4}{1-e\alpha}
\mbox{~~~~for all $\alpha \in (0,1/e)$},
\end{align}
and in particular, 
\begin{equation}
    \label{eq:simple} 
\p(S_K(\mathbf P)\le \alpha)\le \alpha+2\alpha^2+ 6\alpha^3 \mbox{~~~for all $\alpha \in (0, 0.1].$}
\end{equation} 
Since there exists a weakly dependent $\mathbf P\in \mathcal U^K$ (in particular, independent uniforms) such that $\p(S_K(\mathbf P) \le \alpha) = \alpha$, the above bounds are tight up to lower order terms in $\alpha$.
\end{theorem}


Recall from~\eqref{eq:ellK} that since $\p( S_K(\mathbf P)\le\alpha) \leq \ell_K \alpha$ under any dependence structure, the bounds above and below can be improved for small $K$ by taking their minimum with $\ell_K \alpha$, but we often omit this for clarity.
The above bounds also imply the following multiplicative error bounds. 
\begin{corollary}[Multiplicative error bounds]\label{cor:mult-error}
For every weakly negatively dependent   $\mathbf P\in \mathcal U^K$, we have  
\begin{equation}\label{eq:multiplicative-simes-1}
\p( S_K(\mathbf P)\le\alpha) \leq 1.26\alpha \text{~ ~ ~ for all } \alpha \in (0,0.1],
\end{equation}
and also
\begin{equation}\label{eq:multiplicative-simes-2}
\p( S_K(\mathbf P)\le\alpha) \leq 3.4 \alpha \text{~ ~ ~ for all } \alpha \in (0,1),
\end{equation}
meaning that $(3.4 \wedge \ell_K) S_K(\mathbf P)$ is a p-value for any $K$.
\end{corollary}

Before we present the proof, a few comments are in order. For $K=2$, the right hand side of~\eqref{eq:simes-general} becomes $\alpha+\alpha^2$, which equals 0.0525 for $\alpha=0.05$. Despite the theorem holding under weakest form of negative dependence, this value is not so far from the empirically observed value in~\eqref{eq:simes-empirical} for negative Gaussian dependence.
Also, for all practical $\alpha$, the Simes combination results in a valid p-value up to the small constant factor 1.26. However, to formally call it a p-value, the constant is at most 3.4 (though this could potentially be lowered closer to 1 through better approximations).

\begin{proof}[Proof of Theorem~\ref{th:neg} and Corollary~\ref{cor:mult-error}.]
Define $c_k = k\alpha/K$ for  $k\in \mathcal K$.   
Note that for $k\in \mathcal K$, 
$ \{P_{(k)} \leq c_k\}   = \bigcup_{A\in B_k} \bigcap_{j\in A}\{P_{j} \leq c_k\} $, where $B_k=\{A\subseteq \mathcal K:|A|=k\}$ and $|A|$ is the cardinality of $A$.
Bonferroni's inequality gives
\begin{align*}
    \p(S_K(\mathbf P)\le\alpha) 
    &= \mathbb{P}\left(\bigcup_{k=1}^K \{P_{(k)} \leq c_k\}  \right) \\
    &\le \sum_{k=1}^K \mathbb{P}\left(  P_{(k)} \leq c_k  \right) = \sum_{k=1}^K\p  \left(\bigcup_{A\in B_k} \bigcap_{j\in A}\{P_{j} \leq c_k\}  \right).
\end{align*} 
Applying the Bonferroni inequality for every union and \eqref{eq:def-neg}  for every intersection, we  get 
\begin{align}\label{eq:HR}
    \p(S_K(\mathbf P)\le\alpha) 
     &\leq
    \sum_{k=1}^K \sum_{ A\in B_k }\p  \left(\bigcap_{j\in A}\{P_{j} \leq c_k\}  \right)  
   \notag \\
    & \le  \sum_{k=1}^K  \sum_{ A\in B_k }  \prod_{j\in A}\p(P_{j} \leq c_k    )  =   \sum_{k=1}^K\binom{K}{k} c_k^k.
\end{align} 
This shows \eqref{eq:simes-general}.
Note that, for   integers $n\ge k$,
\begin{align*}
    \binom{n}{k} = \frac{n!}{k!(n-k)!} =  \frac{n(n-1)\dots(n-k+1)}{k!} \leq \frac{n^k}{k!}.
\end{align*} 
 Stirling's approximation yields 
\begin{align}\label{eq:BinUbd}
    \binom{n}{k} \leq \frac{n^k}{k!} \leq \frac{n^k}{\sqrt{2\pi k} k^k e^{-k} e^{1/(12k+1)}} \leq \frac{n^k}{\sqrt{2\pi k} k^k e^{-k}}  = \frac{1}{\sqrt{2\pi k}} \left(\frac{en}{k}\right)^k.
\end{align}  
Applying \eqref{eq:BinUbd} to each term of \eqref{eq:HR}  except for the first three terms, we get
\begin{align}
   &\p( S_K(\mathbf P)\le\alpha)   \notag 
    \\&\leq \alpha + \frac{2(K-1)}{K} \alpha^2+
      \frac{9(K-1)(K-2)}{ 2 K^2} \alpha^3 + 
     \sum_{k=4}^ K  \frac{1}{\sqrt{2\pi k}}
    \left(\frac{eK}{k}\right)^k\left(\frac{k}{K} \alpha\right)^k  \notag \\ 
    &= \alpha  +\frac{2(K-1)}{K} \alpha^2+
      \frac{9(K-1)(K-2)}{ 2 K^2} \alpha^3+ \sum_{k=4}^ K \frac{1}{\sqrt{2\pi k}} \left(
   e \alpha\right)^k. \notag
 \\  &  \le \alpha  +2  \alpha^2+
      \frac{9 }{ 2  } \alpha^3+ \sum_{k=4}^\infty \frac{ \left(
   e \alpha\right)^k }{\sqrt{2\pi k}}. \label{eq:sk_neg-use}
\end{align}  
Therefore, by noting that 
$  \sum_{k=4}^\infty \left(
   e \alpha\right)^k    =  {(e\alpha)^4}/({1-e\alpha})
$ for $\alpha < 1/e$,
\eqref{eq:sk_neg-use} implies the inequality \eqref{eq:simes-neg},
and  \eqref{eq:simple} follows from \eqref{eq:simes-neg} by direct computation.

Since any probability is no larger than $1$, an upper bound on the probability in \eqref{eq:simes-neg}  
for all $\alpha\in(0,1)$ is given by the following function 
\begin{align}\label{eq:sk_neg}
    \tilde s_K(\alpha): =\min\left\{\alpha + 2 \alpha^ 2   + \frac{9}{2}\alpha^3 +   \frac{1}{  \sqrt{8\pi}}\frac{(e\alpha)^4}{(1-e\alpha)_+},1\right\},
\end{align}
where $1/0=\infty$ (i.e., the upper bound is $1$ when $\alpha\ge 1/e$).

By \eqref{eq:simple}, $\tilde s_K(\alpha)/\alpha\le 1.26$ 
 for $\alpha \le 0.1$, 
and thus the multiplier to correct for negative dependence is at most $1.26$ for relevant values of $\alpha$. We can also verify $\tilde s_K(\alpha)/\alpha\le 3.4$ for all $\alpha\in (0,1)$.
 \end{proof}

The values of $\tilde s_K(\alpha)$ for common choices of $\alpha\in \{0.01,0.05,0.1\}$, as well as the values of $\alpha$ corresponding to $\tilde s_K(\alpha)\in \{0.01,0.05,0.1\}$, are given in Table~\ref{tab:simes-neg}.  
As we can see from the table, the simple formula \eqref{eq:simple} is a quite accurate approximation of \eqref{eq:simes-neg}.  

\begin{table}[H]
    \centering\renewcommand{\arraystretch}{1.5}

    \begin{tabular}{c|cccccc}
        $ \alpha$  & 0.0098 & 0.01 & 0.0454 & 0.05 & 0.0830  & 0.1    \\\hline
          $\tilde s_K(\alpha)$ & 0.01&  0.0102  & 0.05 & 0.0556 & 0.1 & 0.1260 \\\hline
          $\alpha+2\alpha^2+6\alpha^3$ & 0.0100 & 0.0102 & 0.0501 & 0.0558  & 0.1053 & 0.1260 \\\hline
         $\tilde s_K(\alpha)/\alpha$ &1.020 &  1.020  & 1.101  & 1.112 & 1.205  & 1.260\\\hline
    \end{tabular}
    \caption{Values of the upper bounds in Theorem~\ref{th:neg}}
    \label{tab:simes-neg}
\end{table} 

\begin{remark}
It is clear from the proof of Theorem~\ref{th:neg} that it suffices to require \eqref{eq:def-neg} to hold for $p\in [0,\alpha]$ to obtain  the upper bound in Theorem~\ref{th:neg}. That is, we only need weak negative dependence to hold when all components of $\mathbf P$ are small than or equal to $\alpha$.
\end{remark}

\begin{remark}\label{rem:r2}
Consider a general procedure that rejects the null hypothesis if $p_{(k)}\le t_k(\lambda)$  for some $k\in \mathcal K$,  where $t_k (\lambda) >0$ is a threshold that depends on  a parameter $\lambda.$
The Simes test is a special case that corresponds to $t_k(\lambda) =k\lambda$ with $\lambda = \alpha/K$.
Such procedures are related to the 
joint error rate control of 
\cite{blanchard2020}.
Inspecting the proof of Theorem \ref{th:neg}, we note that \eqref{eq:HR} holds with any choices of $(c_k)_{k\in \mathcal K}$, in particular with $c_k=t_k(\lambda)$. 
Therefore,  the argument for \eqref{eq:sk_neg-use} yields that,  under weak negative dependence, the above procedure has a type-I error rate bounded by 
$$u(\lambda)=Kt_1(\lambda) + \frac{K(K-1)}{2} \left(t_2(\lambda)\right)^2+
      \frac{K(K-1)(K-2)}{ 6 } \left(t_3(\lambda)\right)^3  + 
     \sum_{k=4}^ K  \frac{1}{\sqrt{2\pi k}}
    \left(\frac{eK}{k}\right)^k\left(t_k(\lambda)\right)^k .
    $$
    Choosing $\lambda $ such that $u(\lambda)\le \alpha$ (this is possible if e.g.,  $t_k(\lambda)$ is increasing in $\lambda$ with $t_k(0)=0$ for each $k$) yields a procedure that controls the type-I error at $\alpha$ under weak negative dependence. 
\end{remark}

The analysis in Remark \ref{rem:r2} and \eqref{eq:HR}--\eqref{eq:BinUbd} immediately leads to the following result.
\begin{proposition}\label{prop:ruger}
Fix $k\in \mathcal K$. 
Consider the procedure that rejects the global null if $p_{(k)}\le t$, where $t>0$. If   
$t  \le \left(\alpha/ \eta_{k} \right)^{1/k}$  where $\eta_k={K\choose k}$, then the above procedure has type-I error control at $\alpha$ under  weak negative dependence.
In particular, this holds for
$t  = (2\pi k)^{1/(2k)}    k \alpha^{1/k} /(eK)$,
or  simply
$t  =     k \alpha^{1/k} /(eK).$
\end{proposition}

Recall that the combination method of~\cite{ruger1978} for a fixed $k\in \mathcal K$ rejects the null if $p_{(k)}\le k\alpha/K$ (valid under arbitrary dependence). 
If $k=2$, then the procedure in Proposition \ref{prop:ruger}, rejecting when $p_{(2)}\le  (2\alpha)^{1/2}/K$  (valid under weak negative dependence), 
improves the threshold $ 2\alpha/K$  when $\alpha \le 1/2$.
For general $k$, the procedure is always an improvement over the threshold $ k\alpha/K$  when $\alpha\le e^{-2}$. 

\subsection{Negative Gaussian dependence}
\label{sec:gaussian}

 Theorem~\ref{th:neg} leads to upper bounds on the type-1 error of merging weakly negatively dependent p-values using the Simes test. Below, we discuss the situation of Gaussian-dependent p-values and e-values. 

For a $K\times K$ correlation matrix $\Sigma$,
denote by $\mathcal{G}_\Sigma$ the set of all 
  Gaussian-dependent  random vectors
  with Gaussian correlation $\Sigma$. 
If, $\mathbf X\in \mathcal{G}_\Sigma$ has standard uniform marginals, then its distribution   is called a {Gaussian copula} (see \cite{N06} for copulas).


 The following lemma gives a characterization of a few negative dependence concepts for Gaussian-dependent vectors.

\begin{lemma}\label{lem:gaussian}
For  Gaussian-dependent $\mathbf X\in \mathcal{G}_\Sigma$ with continuous marginals, the following are equivalent:
\begin{enumerate}[label=(\Roman*)]
    \item all off-diagonal entries of $\Sigma$ are non-positive; 
    \item $\mathbf X$ is negatively associated; 
    \item $\mathbf X$ is negative regression dependent; 
    \item $\mathbf X$ is negative orthant dependent; \item$\mathbf X$ is negative lower orthant dependent;  \item $\mathbf X$ is weakly negatively dependent. \end{enumerate}
\end{lemma}
The implications (i)$\Rightarrow$(ii)$\Leftrightarrow$(iii)$\Leftrightarrow$(iv)$\Leftrightarrow$(v)$\Leftrightarrow$(vi) in Lemma~\ref{lem:gaussian} hold true regardless of whether $\mathbf X$ has continuous marginals, but (vi)$\Rightarrow$(i) requires this assumption; a  counterexample is $\mathbf X=(0,\dots,0)$. 

\begin{proof}
Since all statements are invariant under strictly increasing transforms, we
can safely treat $\mathbf X$ as having standard Gaussian marginals (and hence jointly Gaussian).  For Gaussian vectors,  
the implication (i)$\Rightarrow$(ii) is shown by \cite{JP83},
and (i)$\Rightarrow$(iii) is shown by \cite{block1985concept}. 
The implications (ii)$\Rightarrow$(iv)$\Rightarrow$(v)$\Rightarrow$(vi) can be checked by definition; see the diagram in the beginning of Section \ref{sec:defns-negdep}. To see that (vi) implies (i), suppose that an off-diagonal entry $\sigma_{ij}$ of $\Sigma$  is positive (for contradiction).
This implies  $\p(\{X_i\le 0\}\cap \{ X_j\le 0\})>1/4=\p(X_i\le 0)\p(X_j\le 0)$ by direct computation, which violates weak negative dependence. 
\end{proof}

Suppose that $\mathbf X=(X_1,\dots,X_K)$ is negatively Gaussian dependent. 
If a vector of p-values $\mathbf P$   is obtained 
via $\mathbf P=(f_1(X_1),\dots,f_K(X_K))$   for some decreasing functions (or increasing functions) $f_1,\dots,f_K$,
then $\mathbf P$ 
  is also negatively Gaussian dependent; the same applies to a vector of e-values.

 A vector of p-values $\mathbf P\in  \mathcal{G}_\Sigma \cap \mathcal U^K $ is PRD if and only if all entries of $\Sigma$ are non-negative (\cite{BY01}); 
 it is weakly negatively dependent  if and only if all entries of $\Sigma$ are non-positive (Lemma~\ref{lem:gaussian}). 
In the above two cases, the type-1 error of the Simes test applied to $\mathbf P$ is controlled by Proposition~\ref{prop:PRD} and Theorem~\ref{th:neg}. 
 For the intermediate case where some entries of $\Sigma$  are positive and some are negative,
 the type-1 error is much more complicated, and we only have an asymptotic result.
{In the result below, (i) and (iii) are known or follow from existing results, but we briefly explain their proof.}

 \begin{theorem}\label{th:gaussian}
  For Gaussian-dependent $\mathbf P\in \mathcal{G}_\Sigma\cap \mathcal U^K$, the following statements hold. 
  
  \begin{enumerate}[label=(\Roman*)]
  \item  If all  off-diagonal entries of $\Sigma$ are non-negative, then
 $$ 
\p( S_K(\mathbf P)\le\alpha)\le   \alpha 
\mbox{~~~ for all $\alpha \in (0,1)$}
 $$ 
  \item If all off-diagonal entries of $\Sigma$ are non-positive, then 
 $$ 
\p(S_K(\mathbf P)\le \alpha)\le \alpha+2\alpha^2+ 6\alpha^3 \mbox{~~~for $\alpha \in (0, 0.1].$} 
 $$ 
  \item It always holds that
 $$ 
\lim_{\alpha\downarrow 0} \frac{1}{\alpha} \p( S_K(\mathbf P)\le\alpha)\le  1.
 $$
  \end{enumerate}
   \end{theorem}
     The first statement above is well known~\citep{BY01}, but the other two are new.

  \begin{proof}   
For completeness, we point out that statement (i) follows from Proposition~\ref{prop:PRD} and the fact that a Gaussian vector with non-negative pair-wise correlations are PRD. 
Statement (ii) follows from \eqref{eq:simple} and Lemma~\ref{lem:gaussian}. 
To show statement (iii),
define the  harmonic average function $M_{-1,K}:(p_1,\dots,p_K)\mapsto ((p_1^{-1}+\dots+p_K^{-1})/K)^{-1}$. 
 Theorem 2 (ii) of \cite{CLTW22} implies 
$\p(M_{-1,K}(\mathbf P)\le \alpha)/\alpha \to 1$, 
 and Theorem 3 of \cite{CLTW22} gives 
$ M_{-1,K} \le S_K$. Combining these two statements, we get 
$\p(S_K(\mathbf P)\le \alpha) \le \p (M_{-1,K}(\mathbf P)\le \alpha) = \alpha + o(\alpha)$,
thus showing statement (iii).
\end{proof}

Theorem~\ref{th:gaussian} implies, in particular, that the Simes inequality \eqref{eq:simes}  \emph{almost} holds  for all Gaussian-dependent vectors of p-values and $\alpha$ small enough. 
It remains an open question to find a useful upper bound for  $ \p( S_K(\mathbf P)\le\alpha)$
over all Gaussian-dependent  $\mathbf P\in \mathcal U^K$  
for  practical values of $\alpha$ such as $0.05$ or $0.1$. 
A simple conjecture is that \eqref{eq:simple} or a similar inequality holds for all Gaussian-dependent $\mathbf P\in \mathcal U^K$, but a proof seems to be beyond our current techniques. 
{The behaviour of type-I error with varying correlation in the Gaussian setting has been discussed in, for e.g., \citet[Chapter 3]{cui2021handbook}.}

\begin{remark}
Part (iii) of Theorem~\ref{th:gaussian}
holds true under the  condition that 
each pair of components of  
 $\mathbf P$ has a bivariate Gaussian dependence, which is weaker than   Gaussian dependence of $\mathbf P$. The claim follows because this condition is sufficient for Theorem 2 (ii) of \cite{CLTW22}. 
\end{remark}
 
\begin{remark}\label{rem:Gaussian-mix}
Since the probability $\p(S_K(\mathbf P)\le \alpha)$ is linear in (distributional) mixtures, 
the result in Theorem~\ref{th:gaussian} (i) also applies to $\mathbf P$ being a mixture of positively Gaussian-dependent vectors of p-values. Similarly, (ii) applies to  mixtures of negatively Gaussian-dependent  vectors of p-values, and (iii) applies to  mixtures of any Gaussian-dependent ones.
\end{remark}

 \begin{remark}\label{rem:r1-1}
For multivariate Gaussian random vector $(Y_1,\dots,Y_n)$ with standard Gaussian marginals, the vector of the absolute values $(|Y_1|,\dots,|Y_K|)$ is generally not Gaussian-dependent, and hence Theorem \ref{th:gaussian} does not apply to the p-values $(P_1,\dots,P_K)$ computed from  $(|Y_1|,\dots,|Y_K|)$ through $P_k=2\Phi (-|Y_k|)$, $k\in \mathcal K$.
These p-values correspond to two-sided tests for the Gaussian distributions.
A well-known conjecture is that these p-values  satisfy the Simes inequality \eqref{eq:simes}; see e.g., \citet[Section 7]{finner2017simes} for a numerical experiment,
and also Remark \ref{rem:r1-2} for a similar conjecture in the context of FDR control. 

 \end{remark}
 
\subsection{Weighted merging of p-values} 
Let $\mathbf w=(w_1,\dots,w_K)\in \R_+^n$ denote prior weights on the p-values, where $w_1+\dots+w_K=K$; the simplex of such vectors is denoted by $\Delta_K$. These may themselves be obtained by e-values from independent experiments; see \cite{IWR22}, where   the requirement that they add up to $K$ may be dropped (but the terms $\alpha^2$ and $\alpha^3$ will need some correction). 
The weighted Simes function is 
   $$
   S^{\mathbf w}_K(p_1,\dots,p_K)= \bigwedge_{k=1}^K \frac{K}{k} q_{(k)},
   $$
   where $q_k=p_k/w_k$  for $k\in \mathcal K$ and $q_{(1)}\le\dots\le q_{(K)}$ are the order statistics of $q_1,\dots,q_K$.
   Clearly, if $w_1=\dots=w_K=1$, then $  S^{\mathbf w}_K=  S _K$.
    \begin{proposition}
For   weakly negatively dependent  p-values and any $\mathbf w\in \Delta_K$, the bounds in Theorem~\ref{th:neg} hold with $S^\mathbf w_K$ in place of $S_K$.
    \end{proposition}
    \begin{proof}
    It suffices to show that 
     $$ \p(S^\mathbf w_K(\mathbf P)\le\alpha) 
      \le   \sum_{k=1}^K\binom{K}{k} c_k^k$$
      holds, and the remaining steps follow as in the proof of Theorem~\ref{th:neg}.
       Using \eqref{eq:HR} with $P_1,\dots,P_K$ replaced by $P_1/w_1,\dots,P_K/w_K$, we only need to check the inequality in 
     $$  \sum_{ A\in B_k }  \prod_{j\in A}\p(P_{j} \leq w_j c_k    ) =\sum_{ A\in B_k }  \prod_{j\in A} (w_j  c_k) \le     \binom{K}{k} c_k^k,$$
     which holds if  
     \begin{align}\label{eq:weighted-simes}
    \sum_{ A\in B_k }  \prod_{j\in A}  w_j \le { \binom{K}{k} }. 
    \end{align}
    Below we show \eqref{eq:weighted-simes}.  
    Let $W_1,\dots,W_k$ be random samples from $w_1,\dots,w_K$ without replacement.
    By definition, 
we have  $\E[W_1]=\dots=\E[W_k]=1$ and 
$$
     \frac1 { \binom{K}{k} }\sum_{ A\in B_k }  \prod_{j\in A}  w_j
     =\E\left[\prod_{i=1}^k W_i\right ].
     $$
     Since $W_1,\dots,W_k$ are negatively associated (see Section 3.2 of \cite{JP83}), we have 
     $$ 
    \E\left[\prod_{i=1}^k W_i\right ]\le   \prod_{i=1}^k  \E\left[W_i\right ] =1,
     $$
     and hence \eqref{eq:weighted-simes} holds. This is sufficient to obtain the bounds in Theorem~\ref{th:neg}.
    \end{proof}

\subsection{Iterated applications of negative dependence}

A natural question is the following: if $\mathbf{P}$ is negatively dependent, and $A,B$ are two non-overlapping subsets of size $K_1,K_2$, then is it the case that $S_{K_1}(\mathbf P_A)$ and $S_{K_2}(\mathbf P_B)$ are also negatively dependent? 
(In what follows, we suppress the subscripts 
 $K_1$ and $K_2$ for readability.)
We cannot settle this question for all definitions of negative dependence, but we can prove the following.

\begin{proposition}\label{prop:simes-NA-to-NOD}
    If $\mathbf{P}$ is negatively associated, and $\{A_k\}_{k=1,\dots,\ell}$ are non-overlapping subsets of $\mathcal K$, then $(S(\mathbf P_{A_1}),\dots, S(\mathbf P_{A_\ell}))$ is negative  orthant dependent.
    The same result holds for any monotone p-value combination rule (such as Fisher's, Stouffer's or Bonferroni, median, average, etc.). 
\end{proposition}

\begin{proof}
Recall the implication of negative 
association~\eqref{eq:na-implication}. 
For arbitrary constants $s_1,\dots,s_\ell\ge 0$, choose the coordinatewise increasing nonnegative functions $\phi_k$ as $\id_{\{S(\mathbf P_{A_k}) > s_k\}}$ to yield
\[
\p( S(\mathbf P_{A_1}) > s_1, \dots, S(\mathbf P_{A_\ell}) > s_\ell) \leq \prod_{k=1}^\ell  \p(S(\mathbf P_{A_k}) > s_k),
\]
and this shows negative upper orthant dependence of $(S(\mathbf P_{A_1}),\dots, S(\mathbf P_{A_\ell}))$. 
To obtain negative lower orthant dependence, it suffices to note that 
\eqref{eq:def-na} holds for componentwise decreasing $f,g$, 
and thus
\eqref{eq:na-implication} also holds for non-negative componentwise decreasing functions $\phi_k$ chosen as $\id_{\{S(\mathbf P_{A_k})\le s_k\}}$. 
\end{proof}

The above proposition proves useful in group-level FDR control, as we shall see later. For now, we describe an implication for global null testing with grouped hypotheses in  the following corollary.  
\begin{corollary}\label{cor:simes-of-simes}
    If $\mathbf{P}$ is negatively associated, and $A_1,\dots,A_\ell$ are non-overlapping subsets of $\mathcal K$, 
    then 
    \[
    \p\Big( S( S(\mathbf P_{A_1} ),\dots, S(\mathbf P_{A_\ell})) \leq \alpha \Big) \leq 1.52 \alpha \text{ ~ ~ for all } \alpha \in [0, 0.083),
    \]
    and also
    \[
     \p\Big( S( S(\mathbf P_{A_1} ),\dots, S(\mathbf P_{A_\ell})) \leq \alpha \Big) \leq (3.4 \wedge \ell_K )^2 \alpha \text{ ~ ~ for all } \alpha \in [0,1],
    \]
    meaning that $(3.4 \wedge \ell_K )^2 S( S(\mathbf P_{A_1} ),\dots, S(\mathbf P_{A_\ell}))$ is a valid p-value. 
\end{corollary}

In the first inequality in Corollary~\ref{cor:simes-of-simes},
the values $0.083$ and $1.52$ are computed from Table~\ref{tab:simes-neg} by applying \eqref{eq:multiplicative-simes-1} twice. 
The second inequality is due to the fact that $3.4S(\mathbf P)$ is a p-value by Theorem~\ref{th:neg} under weak negative dependence.

  In contrast to Corollary~\ref{cor:simes-of-simes}, if $\mathbf{P}$ is positively regression dependent (PRD), then
    \[
    \p\Big(S( S(\mathbf P_{A_1} ),\dots, S(\mathbf P_{A_\ell})) \leq \alpha \Big)   \leq  \alpha.
    \]
This inequality was proved by~\cite[Lemma 2(d)]{ramdas2019unified}. Surprisingly, it holds despite the fact that $S(\mathbf P_{A_1}),\dots, S(\mathbf P_{A_\ell})$ are not known to themselves be PRD (even though $\mathbf{P}$ is); in fact, the claim under PRD even holds for overlapping groups. It is likely that under certain types of mixed dependence (such as positive dependence within groups but negative dependence across groups, or vice versa), intermediate bounds can be derived.

\section{Merging negatively dependent e-values}\label{sec:merge-e}

E-values (\cite{VW21}) are an alternative to p-values as a measure of evidence and significance. 
We make a brief but important observation on negatively associated e-values.
An e-variable (also called an e-value, with slight abuse of terminology) for 
testing a hypothesis $H$ 
is a random variable $E\ge 0$ with $\E^Q[E]\le 1$ for each probability measure $Q\in H$. Further, recall that an e-value may be obtained by \emph{calibration} from a p-value $P$, i.e., $E=\phi(P)$ for some
 \emph{calibrator} $\phi$, which is a nonnegative decreasing function $\phi$ satisfying $\int_0^1 \phi (t) \d t\le 1$ (typically with an equal sign); see \cite{shafer2011test} and \cite{VW21}.

\begin{theorem}\label{prop:nuod-e}
If e-values $E_1,\dots,E_K$ are negatively upper orthant dependent, then $\prod_{i=1}^k E_i$ is also an e-value for each $k\in\mathcal K$. More generally, $E(\boldsymbol \lambda) := \prod_{i=1}^K (1-\lambda_i + \lambda_iE_i)$ is an e-value for any constant vector $\boldsymbol \lambda:=(\lambda_1,\dots,\lambda_K)\in [0,1]^K$. 
In particular, if the e-values are obtained by calibrating negative lower orthant dependent p-values, then they are negatively upper orthant dependent.
\end{theorem}
The above proposition is recorded for ease of reference, but its proof is simply a direct consequence of \eqref{eq:nuod-f}. 
The condition of negative upper orthant dependence in Theorem~\ref{prop:nuod-e} is weaker than negative orthant dependence or negative association. 
Thus if $\mathbf P$ is Gaussian dependent,  and all  off-diagonal entries of $\Sigma$ are non-positive, then
 $
 E := \prod_{k=1}^K  \phi_k(P_k)  
 $ 
 is an e-value 
 for any calibrators $\phi_1,\dots,\phi_K$.
 
 Products are not the only way to combine negatively dependent e-values. The next proposition lays out certain admissible combinations.


\begin{corollary}\label{cor:U}
For negatively upper orthant dependent e-values $E_1,\dots,E_K$, 
 convex combinations of  $$\prod_{k\in A} E_k,\mbox{~where $A\subseteq \mathcal K$,}$$
are also valid e-values (here the product is $1$ if $A=\varnothing$). This family includes U-statistics of $E_1,\dots,E_K$. Further, such convex combinations, treated as functions from $[0,\infty)^K\to [0,\infty)$, are admissible merging functions for negative orthant dependent e-values.
\end{corollary}

 
The validity follows because averages of arbitrarily dependent e-values are always e-values.
The admissibility follows because these merging functions are admissible within the larger class of merging functions for independent e-values (see \cite{VW21}).

For independent e-values,
 \cite{vovk2020true} observed from simulations that the U-statistic
\begin{align}
U_2:=\frac{2}{K(K-1)} \sum_{1\le k<j\le K} E_k E_j 
\label{eq:U2}
\end{align}
performs quite well in some numerical experiments. Similarly, 
\begin{align}
U_3:=\frac{6}{K(K-1)(K-2)} \sum_{1\le k<j<\ell \le K} E_k E_jE_\ell 
\label{eq:U3}
\end{align}
can be useful in different situations. 
Since $U_2$ and $U_3$ are both valid e-values under negative upper orthant dependence, we will use these e-values in our simulation  examples below.  


Note that the Simes combination for e-values, given by 
\[
   S_K(e_1,\dots,e_K)= \bigvee_{k=1}^K \frac{k}{K} e_{[k]}, \mbox{~where $e_{[k]} $ is the $k$-th largest order statistic of $e_1,\dots,e_K$,}
\] 
does result in a valid e-value under arbitrary dependence, but it is uninteresting because it is dominated by the average of the e-values, which is also valid under arbitrary dependence as mentioned above. Thus we only discuss Simes in the context of p-values in this paper.

We end this subsection by presenting an important corollary of Theorem~\ref{prop:nuod-e} that pertains to the construction of particular e-values that are commonly encountered in nonparametric concentration inequalities. To set things up, following~\cite{boucheron2013concentration}, we call a mean-zero random variable $X$ as $v$-sub-$\psi$, if the following condition holds: for every $\lambda \in \mathrm{Domain}(\psi)$, $\mathbb{E}[e^{\lambda X}] \leq e^{\psi(\lambda)v}$, which is simply a bound on its moment generating function. If $X$ is not mean zero, then it is called $v$-sub-$\psi$ if the aforementioned condition is satisfied by $X - \mathbb{E}[X]$. In particular, if $\psi(\lambda)=\lambda^2/2$, then $X$ is called $v$-subGaussian. 
In what follows, we use $n$ instead $K$ for the number of random variables involved, as $n$ here often corresponds to the number of observations instead of the number of tests.

\begin{corollary}[Chernoff e-variables]\label{cor:chernoff}
Suppose $X_1,\dots,X_n$ are negatively associated, and that each $X_i$ is $v_i$-sub-$\psi_i$. Then, denoting by $\mu_i:= \mathbb{E}[X_i]$, we have that $\exp( \sum_{i=1}^n \lambda_i (X_i-\mu_i) - \sum_{i=1}^n \psi_i(\lambda_i) v_i)$ is an e-value for any positive constants $\lambda_1,\dots,\lambda_n$ in the domains of $\psi_1,\dots,\psi_n$, respectively.
\end{corollary}

If the sub-$\psi$ condition holds only for some subset $\Lambda \subseteq \mathrm{Domain}(\psi)$, then so does the final conclusion. The proof follows directly by invoking Corollary~\ref{cor:U} with the e-values $E_i := \exp(\lambda_i(X_i - \mu) - \psi_i(\lambda_i)v_i)$.

Such e-values appeared centrally in the unified framework for deriving Chernoff bounds in \cite{howard2020time}, and thus we call them Chernoff e-values. As a particular example, assume that for all $i$, we have $\mu_i=\mu,~ v_i=v$ and $\psi_i(\lambda)=\lambda^2/2$, and we also choose $\lambda_i=\lambda$. Denoting $\hat \mu_n := \sum_{i=1}^n X_i/n$, we obtain that $\exp(n\lambda (\hat \mu_n - \mu) - nv\lambda^2/2)$ is an e-value. Applying Markov's inequality, we see that $$\mathbb P\left(\hat \mu_n - \mu > \frac{\log(1/\alpha)}{n\lambda} + v\lambda/2 \right) \leq \alpha.$$ Choosing $\lambda = \sqrt{2\log(1/\alpha)/(nv)}$, we get  ``Hoeffding's inequality'' for averages of subGaussian random variables: $\mathbb P(\hat \mu_n - \mu > \sqrt{{2v\log(1/\alpha)}/{n}} ) \leq \alpha$, which is known to hold under negative association.

Other examples of this type can be derived; see for instance Example 2 in Appendix A.1 and \cite[Fact 1 and Lemma 3]{howard2020time}. 

 \section{False discovery rate control}\label{sec:fdr}
 
\subsection{The BH procedure}
\label{sec:BH-proc} 
In this section, we present an implication of our results in controlling the FDR. We will obtain an FDR upper bound that may not be very practical. Nevertheless, it is the first result we are aware of that controls FDR under negative dependence (without the Benjamini-Yekutieli corrections of $\approx \log K$~\citep{BY01}), and hence it represents an important first step that we hope open the door to future work with tighter bounds.

Let $H_1,\dots,H_K$ be $K$ hypotheses.
For each $k\in \mathcal K$,  $H_k$ is called a true null if $\p \in H_k$. 
Let $\mathcal N\subseteq \mathcal K$ be the set of indices of true nulls, which is unknown to the decision maker, and $K_0$ be the number of true nulls, thus the cardinality of $\mathcal N$.  
   For each $k\in \mathcal K$, 
  $H_k$ is  associated with p-value $p_k$, which is a realization of a random variable $P_k$. If $k\in \mathcal N$, then $P_K$ is a p-variable, assumed to be uniform under $[0,1]$. 
 We write the set of such $\mathbf P$ as $\mathcal U^K_{\mathcal N}$.
We do not make any distribution assumption on $P_k$ for $k\in \mathcal K\setminus \mathcal N$. 
 
   
 
A random vector $\mathbf P$ of p-values is \emph{PRD on the subset $\mathcal N$} (PRDS) if for any null index $k\in \mathcal N$ and   increasing set $A  \subseteq  \R^K$, the
function $x\mapsto \p(\mathbf P\in A\mid P_k\le x)$ is increasing on $ [0,1]$. 
If $\mathcal N=\mathcal K$, i.e., all hypotheses are null, then PRDS is precisely PRD. 
For a Gaussian-dependent random vector $\mathbf P$ with Gaussian correlation matrix ${\Sigma}$, it is PRDS if and only if 
$\Sigma_{ij}\ge 0$ for all $i\in \mathcal N$ and $j\in \mathcal K$.


A  testing procedure $\cD:[0,1]^K\to 2^{\mathcal K}$ 
reports rejected hypotheses (called discoveries) based on observed p-values. 
We write $F_{\cD}$ as the number of null cases that are rejected (i.e., false discoveries), and $R_{\cD}$ as the total number of discoveries truncated below by $1$, that is,
$$
F_{\cD} =   | \cD(\mathbf P) \cap \mathcal N| \mbox{~~~and~~~} R_{\cD} =|\cD(\mathbf P)|\vee 1.
$$ 
 The value of interest is   $F_{\cD}/R_{\cD}$, called the false discovery proportion (FDP), which is the ratio of the number of false discoveries to that of all claimed discoveries, with the convention $0/0=0$ (i.e., FDP is $0$ if there is no discovery; this is the reason of truncating $R_{\cD}$ by $1$). 
 \cite{BH95} introduced the FDR, which is the expected value of FDP, that is,
  $$ 
 \mathrm{FDR}_{\cD} =\E \left[ \frac{F_{\cD}}{R_{\cD}  } \right],
$$
where the expected value is taken  under the true probability.    
The \emph{BH procedure} $\cD_\alpha$ of \cite{BH95} rejects all hypotheses with the smallest $k^*$ p-values,
where 
 $$
k^*=\max\left\{k\in \mathcal K: \frac{K p_{(k)}}{k} \le \alpha\right\},
 $$
with the convention $\max(\varnothing)=0$, 
and accepts the rest.
For independent (\cite{BH95}) or PRDS (\cite{BY01}) p-values, the BH procedure has an FDR guarantee \begin{align}\label{eq:bh} 
\E\left[\frac{F_{\cD_\alpha}}{R_{\cD_\alpha}   }\right] \le \frac{ K_0}{K}\alpha 
\mbox{~~~ for all $\alpha \in (0,1)$}.\end{align}  
\begin{proposition}[\cite{BY01}]\label{pr:by}
If the  vector of p-values $\mathbf P$ is PRDS, then \eqref{eq:bh} holds. For arbitrarily dependent p-values, the error bound in \eqref{eq:bh} is multiplied by $\ell_K$, similar to \eqref{eq:simes3}.
\end{proposition}


As a consequence of Proposition~\ref{pr:by}, for Gaussian-dependent $\mathbf P$, \eqref{eq:bh}   holds when the correlations are non-negative.   
In   the setting that all hypotheses are true nulls, i.e., $K_0=K$,  it holds that 
     $$
\E\left[\frac{F_{\cD_\alpha}}{R_{\cD_\alpha}  }\right] = 
 \p( | \cD_\alpha( \mathbf P) |>0)
 =
 \p\left(\bigcup_{k\in \mathcal K}\left\{ \frac{K P_{(k)}}{k} \le \alpha\right\}\right)
=\p(S_K(\mathbf P)\le \alpha).
$$
Hence, in this setting, the FDR is equal to $\p(S_K(\mathbf P)\le \alpha)$, and \eqref{eq:bh} becomes \eqref{eq:simes}.  
If $P_1,\dots,P_K$ are independent, 
and the null p-values are uniform on $[0,1]$, then \eqref{eq:bh} holds as an equality,  similar to \eqref{eq:simes2}.

\subsection{FDR control under negative dependence}

We provide an upper bound on the FDR of the BH procedure for weakly negatively dependent p-values, that shows that the error inflation factor is  independent of $K$ (unlike Proposition~\ref{pr:by}). The proof is based on some interesting results in a preprint by \cite{S18}.

\begin{theorem}\label{thm:Su-neg}
If the null p-values are weakly negatively dependent (and allowing for arbitrary dependence between nulls and non-nulls), then the BH procedure at level $\alpha$ has FDR at most \[\phi (\alpha,K):=\alpha ( (  - \log \alpha  + 3.18   ) \wedge \ell_K).\] 
The bound is asymptotically tight in the following sense:
Fix $\epsilon\in (0,1)$.
For all $\alpha>0$ small enough (depending on $\epsilon$) and $K$ 
large enough (depending on $\alpha$ and $\epsilon$), there exist iid uniform null p-values and some non-null p-values such that the BH procedure has an FDR at least $(1-\epsilon) \phi (\alpha,K)$.
\end{theorem}
Technically, the above bound can be improved to $\alpha ((  - \log \alpha  + 3.18   ) \wedge \ell_K \tfrac{K_0}{K})$, but since $K_0$ is unobservable, we omit it above for simplicity. The $\ell_K$ multiplier is slightly tighter for small $K$ and $\alpha$, but obviously the overall bound still does not grow with $K$.

\begin{proof}
The $\ell_K\alpha$ bound holds by Proposition~\ref{pr:by}, so we ignore it.
 Theorem 1 of \cite{S18} yields 
 \begin{align}\label{eq:Su_bd}
     \mathrm{FDR}_{\cD_\alpha} \le \alpha + \alpha \int_\alpha^1 \frac{\mathrm{FDR}_0(x)}{x^2} \d x,
 \end{align}
 where $\mathrm{FDR}_0(x)$ is the FDR of the BH procedure applied to only the null p-values at level $x$.
We will apply the upper bound on $\mathrm{FDR}_0(x) $ obtained from Theorem~\ref{th:neg}. 
We assume $\alpha\le0.3$, because there is nothing to show for the case $\alpha>0.3$ in which  the claimed FDR upper bound is larger than $1$. 
Let $\alpha_0=0.3$, which is chosen to be close to $1/e$. Note that
\begin{align*}
    \int_\alpha^{\alpha_0} \frac{1}{x^2} \sum_{k=4}^\infty \frac{\left(e x\right)^k}{\sqrt{2\pi k}}  \d x
    &= \int_\alpha^{\alpha_0}  \sum_{k=4}^ \infty   \frac{1}{\sqrt{2\pi k}} e^k x^{k-2} \d x\\
    &= \sum_{k=4}^ \infty  \frac{1}{k-1} \frac{1}{\sqrt{2\pi k}}e^k x ^{k-1} |_{\alpha}^{\alpha_0} \le \sum_{k=4}^ \infty \frac{1}{k-1} \frac{1}{\sqrt{2\pi k}}e^k x ^{k-1} |_{0}^{0.3}  \approx 0.2473.
\end{align*}  
By applying \eqref{eq:sk_neg-use} to \eqref{eq:Su_bd}, and using the above upper bound, we get
\begin{align*}
      \mathrm{FDR}_{\cD_\alpha} &\le \alpha + \alpha \int_\alpha^1 \frac 1{x^2} \min\left\{ \left( x  +2x^2 +
      \frac{9}2 x^3+ \sum_{k=4}^\infty  \frac{ \left(
   e x\right)^k  }{\sqrt{2\pi k}}   \right),1 \right\}\d x \\
   &\le \alpha + \alpha \left(\int_\alpha^{\alpha_0} \frac 1{x^2}  \left( x  +2x^2 +
      \frac{9}2 x^3+ \sum_{k=4}^ \infty  \frac{ \left(
   e x\right)^k  }{\sqrt{2\pi k}}\right)\d x + \int_{\alpha_0}^1 \frac 1{x^2} \d x\right)\\
   &\le  \alpha +  \alpha  \left(\int_\alpha^{\alpha_0} \frac 1{x} \d x +    \int_0^{\alpha_0}    \left(2+\frac 9 2 x \right) \d x  + 0.2474 + \int_{\alpha_0}^1 \frac1 {x^2} \d x   \right) \\
   &\le \alpha +  \alpha  \left( \log  \alpha_0 -\log \alpha  +  1.05  + (1/\alpha_0-1)  \right)
    \le  \alpha  \left(  -\log \alpha  + 3.1792  \right),
\end{align*}
and this gives the stated upper bound. 

Next, we show sharpness. 
If the null p-values are iid uniform on $[0,1]$ (but allowing for arbitrary dependence between nulls and non-nulls),  Theorem 3 of \cite{S18} gives an upper bound 
$$
     \mathrm{FDR}_{\cD_\alpha} \le \alpha (-\log \alpha +1).
$$ 
A closer examination of the proof of Theorem 4 of \cite{S18} shows that the FDR in this case (that is, with iid uniform nulls) can be at least $(1-\epsilon) \alpha(-\log \alpha +1)$ for $\alpha$ small enough and $K$ large enough.
Therefore,  it suffices to notice that our bound
$\phi (\alpha,K)=\alpha ( (  - \log \alpha  + 3.18   ) \wedge \ell_K)$
and \cite{S18}'s bound 
$\alpha   (  - \log \alpha  + 1  )$ are asymptotically equivalent for $\alpha $ sufficiently small
and $K$ sufficiently large.
\end{proof}

Although the sharpness statement in Theorem \ref{thm:Su-neg} implies that the bound there cannot be improved essentially, 
it is unclear whether the bound can be improved
if we further assume that all p-values (nulls and non-nulls) are negatively dependent. We leave this as an open question.

 Note that for $K=2$, the Simes error bound was $\alpha+\alpha^2$, and so \eqref{eq:Su_bd} gives
 $
 \mathrm{FDR}_{\cD_\alpha} \le  \alpha (2- \alpha - \log\alpha).
 $
For $\alpha\le 1/2$, this bound is weaker than that of \cite{BY01}, which gives $1.5\alpha$.

The $\mathrm{FDR}$ upper bound in Theorem~\ref{thm:Su-neg} (ignoring the $\ell_K$ term) for some $\alpha$   are given in Table~\ref{tab:Su-neg}.  

\begin{table}[H]
    \centering\renewcommand{\arraystretch}{1.5}
    \begin{tabular}{c|cccccc}
        $ \alpha$  & 0.00603 & 0.01 & 0.01334 & 0.05 &  0.1    \\\hline
         $\text{FDR}$ & 0.05 &  0.07784 & 0.1  &  0.3087 & 0.54812 \\\hline
         $\text{FDR}/\alpha$ & 8.292 & 7.784 & 7.502 & 6.175 & 5.482\\\hline 
    \end{tabular}
    \caption{Values of the FDR upper bounds in Theorem~\ref{thm:Su-neg}.}
    \label{tab:Su-neg}
\end{table}

As seen from Table~\ref{tab:Su-neg}, the upper bound produced by Theorem~\ref{thm:Su-neg} can be   quite conservative in practice, although it is better than the $\ell_K$ correction of \cite{BY01} for large $K$.

A remaining open question is to find a better FDR bound with stronger conditions of negative dependence.
On the other hand, the e-BH procedure (\cite{WR22}) controls FDR for arbitrarily dependent e-values, which we will compare with in our simulation results. 

\begin{remark}[Two-sided Gaussain BH conjecture]\label{rem:r1-2}
{Consider the two-sided Gaussian p-values in Remark \ref{rem:r1-1}, i.e., $P_k=2\Phi(-|Y_k|)$ for $k\in \mathcal K$ where $(Y_1,\dots,Y_k)$ is multivariate Gaussian with standard Gaussian marginals and an arbitrary correlation matrix. 
A folklore conjecture is that for such p-values, the BH procedure at level $\alpha$ controls FDR at $\alpha$.
See \citet[Chapter 4]{roux2018inference} for a discussion. 
This conjecture holds true in the case of $K=2$, as analyzed by \cite{reiner2007fdr}.}
\end{remark}

\subsection{Group-level FDR control}

Sometimes, data are  available at a higher resolution (say single nucleotide polymorphisms along the genome, or voxels in the brain), but we wish to make discoveries at a lower resolution (say at the gene level, or higher level regions of interest in the brain). 
This leads to the question of group-level FDR control \citep{ramdas2019unified}. 
The $K$ hypotheses are divided into $G$ groups. We have p-values for the $K$ individual hypotheses, but wish to discover groups that have some signal without discovering too many null groups (a group is null if all its hypotheses are null, and it is non-null otherwise). In other words, we wish to control the group-level FDR with hypothesis-level p-values.

A natural algorithm for this is to combine the p-values within each group using, say, the Simes combination, and then apply the BH procedure to these group-level ``p-values''. We use ``p-values'' in quotations because while the Simes combination does lead to a p-value under positive dependence (PRDS), as we have seen it only leads to an approximate p-value if the p-values are negatively dependent. Let us call this the Simes+BH$_\alpha$ procedure; to clarify, it applies the BH procedure at level $\alpha$ to the group-level Simes ``p-values'' formed by applying the Simes combination to the p-values within each group, without any corrections. Then we have the following result.

\begin{proposition}\label{prop:BH-Simes} If the p-values are negatively associated, the Simes+BH$_\alpha$ procedure controls the group-level FDR at level $3.4 \alpha(-\log(3.4\alpha)+3.18)$. 
\end{proposition}
\begin{proof} {By Corollary \ref{cor:mult-error}, for negatively associated p-values $\mathbf P$, their Simes combination $S_K(\mathbf P)$ multiplied by 3.4 is a p-value. 
Moreover, the p-values resulting from applying the Simes combination to each group are negative orthant dependent by Proposition~\ref{prop:simes-NA-to-NOD}.  
Running the BH$_\alpha$ procedure on these  p-values is equivalent to running the BH$_{3.4\alpha}$ procedure on the  ones corrected by multiplying 3.4. 
Then, the FDR control follows from  Theorem \ref{thm:Su-neg}.} 
\end{proof}

{In contrast to Proposition \ref{prop:BH-Simes}, 
if the p-values are PRDS, the Simes+BH$_\alpha$ procedure controls the group-level FDR at level $\alpha$; this follows as a direct consequence of results in~\citet{ramdas2019unified}.} 
As earlier in the paper, both instances of $3.4$ in Proposition \ref{prop:BH-Simes} can be replaced by $3.4 \wedge \ell_K$, which is  tighter for small $K$, but it has been omitted for clarity.

The FDR bound in Proposition~\ref{prop:BH-Simes}, due to repeatedly applying bounds under negative dependence, may be quite conservative in practice. Nevertheless,  it is the first result on the group-level FDR control under negative association which does not has an exploding penalty term (compared to the classic BH procedure) as $K\to\infty$, similarly to 
 the case of Theorem~\ref{thm:Su-neg}. 
Future studies may improve this bound. 

\smallskip

As a different grouped setting to consider, suppose we want to control the FDR at the level of the individual hypotheses, but the p-values happen to be divided into $G$ groups such that they  are independent across groups and negatively dependent within each group. Then at least two options exist. One option is to apply the BH procedure to all hypotheses at once (correcting for negative dependence). Another option is to divide the p-values into $K/G$ independent sets (each containing at most one hypothesis from each group), and apply the BH procedure at level $\alpha G/K$ separately to each set, and simply take the union of all discoveries made. It is likely that neither method uniformly dominates the other, and their relative performance will depend on $G,K$ and the type of dependence.

\section{Simulations}\label{sec:sims}

We now apply our results in Theorems~\ref{th:neg},~\ref{prop:nuod-e} and~\ref{thm:Su-neg} to   global testing or multiple testing problems, and illustrate them by means of several simulation experiments.

The test statistics $X_k,~k\in\mathcal K$  are generated  from correlated z-tests, and they are jointly Gaussian. 
The null hypotheses are $ \mathrm{N}(0,1)$ and the alternatives are $ \mathrm{N}(\delta,1)$,
where $\delta\ge 0$.  
Among the $K$ test statistics, $K_0$ of them are drawn from the   null hypothesis and $K-K_0$ of them are from the alternative hypothesis. Let $\pi_0= K_0/K$ be the proportion of true null hypotheses. 

We compute the p-values 
\(
  P(x)
  :=1-
  \Phi(x),
\)
for $x  \in \{X_k: k\in \mathcal K\}$;
which are based on the most powerful test (Neyman--Pearson lemma). 
To compute e-values, we first compute the likelihood ratio:
 $$
  L_\delta (x)
  :=
  \frac{\exp(-(x-\delta)^2/2)}{\exp(-x^2/2)}
  =
  \exp(\delta x - \delta^2/2)
 $$
of the alternative to the null density (which are obviously unit mean under the null, and hence e-values). 
Since $\delta$ may not be known to the tester, we  take an average of $L_\delta(x)$ with respect to $\delta$ on $[a,b]$,
that is,
\begin{align}\label{eq:compute-e}
E(x): = \frac{1}{b-a}\int_{a}^b L_\delta(x) \d \delta = {\sqrt{{2\pi}} \exp(x^2/2) \left(\Phi(x-a) - \Phi (x-b) \right)} .
\end{align}
Since mixtures of e-values are also e-values, the above is also an e-value.
Note that the validity of the p-values and e-values defined above only depends on the null hypothesis but not on the alternative hypothesis (but the power depends on reasonably accurate specification of the alternative). 
As mentioned before, if $(X_1,\dots,X_K)$ is  negatively Gaussian dependent, 
then the p-values, as componentwise increasing transforms of $(X_1,\dots,X_K)$,   are  negatively Gaussian dependent. 
Similarly, the e-values, as  componentwise decreasing transforms of $(X_1,\dots,X_K)$,  are also  negatively Gaussian dependent. 

\subsection{Testing a global null}
\label{sec:62}
We first run simulations to test 
a global null with various methods combining p-values or e-values. 
We are mainly interested in the methods in this paper under negative dependence, and we will also look at their performance under positive dependence for a comparison. That is, 
we consider the following settings, and each simulation will be repeated 10,000 times  and we report their average. 
\begin{enumerate}
    \item Negative Gaussian dependence: Set the pairwise correlation of $(X_1,\dots,X_K)$ to be uniformly chosen between $[-1/(K-1),0]$ in each simulation.\footnote{Effectively, we are simulating from mixtures of negatively Gaussian-dependent p-values; see Remark \ref{rem:Gaussian-mix}. Recall that $-1/(K-1)$ is the smallest possible value for the pairwise correlation coefficients of an exchangeable random vector.}
    For this setting, we 
    {provide a few different scenarios of $(K,\pi_0)$: we take $\pi_0\in\{0, 50\%, 90\%\}$, corresponding to full signal, rich signal, and sparse signal,
    and fix $K=100$ in the main text. Additional simulation results for $K\in \{10,1000\}$, corresponding to a smaller pool and a larger pool of hypotheses, are provided in Appendix.}
    \item Positive Gaussian dependence: Set the pairwise correlation of $(X_1,\dots,X_K)$ to be uniformly chosen between $[0,1]$ in each experiment. For this setting, we consider two different scenarios of $(K,\pi_0)$: $K=1000$ and $\pi_0\in\{0,50\%\}$, to illustrate some simple comparative observations.
\end{enumerate}
We let $\delta$ vary in $[0,3]$ and the e-values in \eqref{eq:compute-e} will be computed with averaging $\delta\in [a,b]=[0,3]$. Fix the type-I error upper bound as $0.05$.
We consider the following seven methods:
\begin{enumerate}[label=(\alph*)]
\item the Simes method with $\ell_K$ correction in \eqref{eq:simes3};
\item the Bonferroni correlation (the minimum of p-values times $K$); 
\item the Simes method with negative dependence (ND) correction (first row of Table~\ref{tab:simes-neg}); 
 \item 
 arithmetic mean of of e-values; 
 \item order-2 U-statistic (U2) of e-values in \eqref{eq:U2};
 \item order-3 U-statistic (U3) of e-values  in \eqref{eq:U3};
 \item the product e-value.
 \end{enumerate} 
 All methods based on e-values are compared against the threshold $20$, so that they have a type-I error guarantee of $0.05$.
 Among these methods, (a), (b) and (d) are valid under arbitrary dependence structures. Method (c) is valid under both negative Gaussian dependence (Theorem~\ref{th:neg}) and positive Gaussian dependence (implying PRD in Proposition~\ref{prop:PRD}). The remaining methods (e), (f) and (g) are valid under negative orthant dependence (Theorem~\ref{prop:nuod-e}).
 
We  plot the rejection probabilities of the above methods in Figure~\ref{fig:globalnull_neg} 
for the setting of negative Gaussian  dependence
and in Figure~\ref{fig:globalnull_pos}
for the setting of positive Gaussian dependence. 

\begin{figure}[t]
    \centering
    \caption{Global null testing (negative Gaussian dependence). All three subplots show power against $\delta$. The left endpoint of $\delta=0$ actually represents the achieved type I error, which appears to be controlled at the nominal level $\alpha=0.05$ for all methods. The key observations are in the text.}
    \label{fig:globalnull_neg}
    \includegraphics[width=0.99\textwidth]{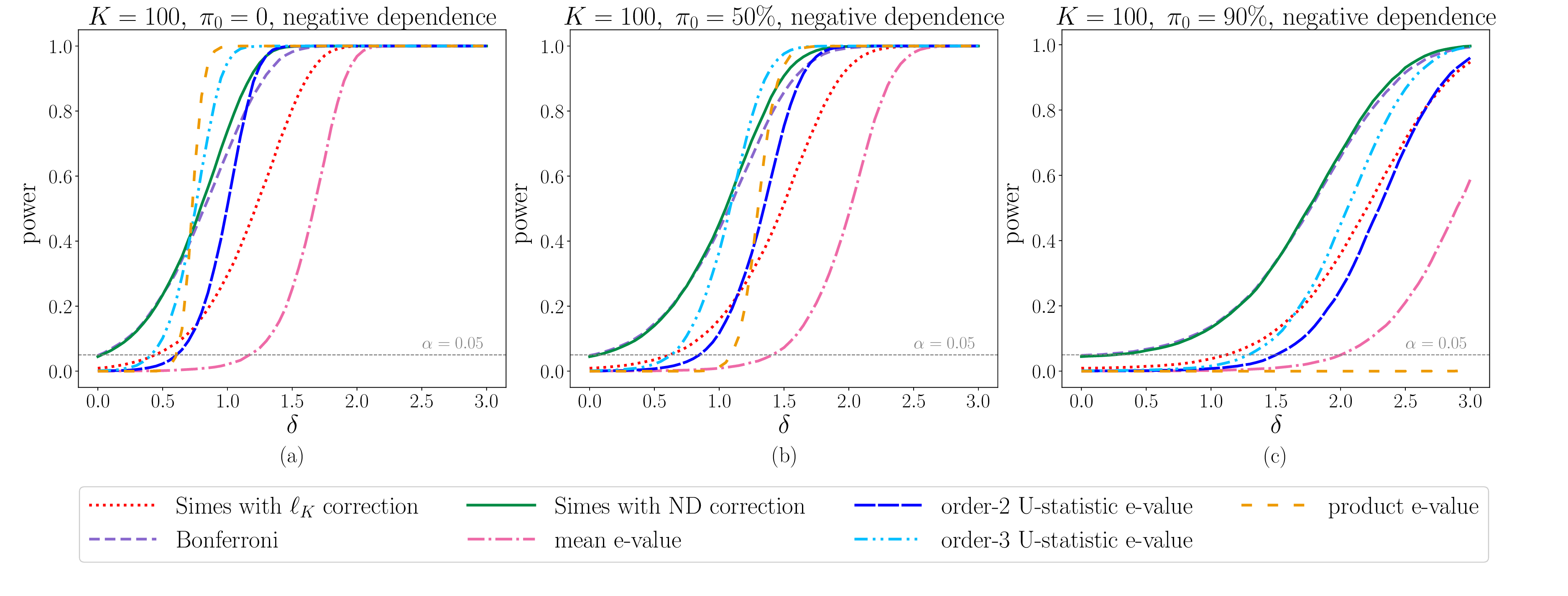}
    \vspace{-6mm}
\end{figure}

\begin{figure}[t]
    \centering
    \caption{Global null testing (positive Gaussian dependence)}
    \label{fig:globalnull_pos}
    \includegraphics[width=0.99\textwidth]{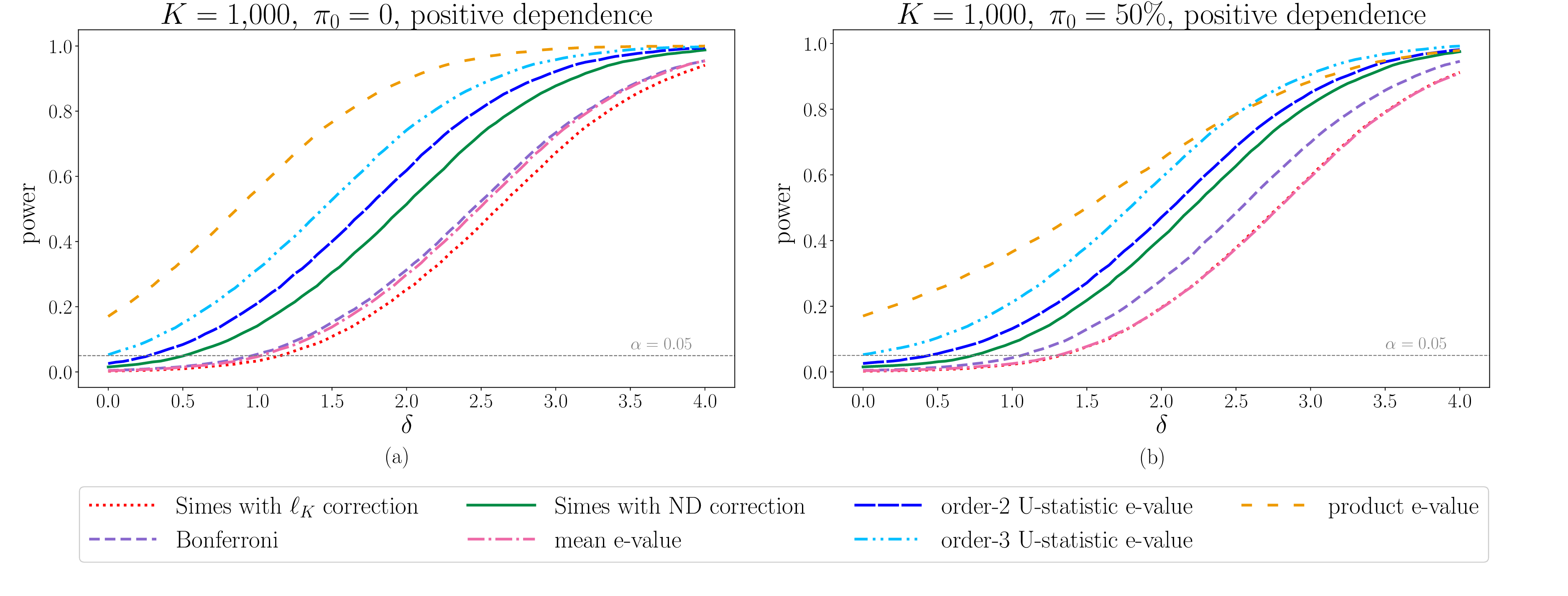}
    \vspace{-0.95cm}
\end{figure}

From Figure~\ref{fig:globalnull_neg}, observe that 
the product e-value has strong power when the signal is full ($\pi_0=0$), but its performance reduces substantially when the signal is rich ($\pi_0=50\%$) or sparse ($\pi_0=90\%$). 
The U3 e-value performs quite well when the signal is full or rich, and loses power when signal is sparse, and the U2 e-value is similar to the U3 e-value with lower power. 
The Simes method with ND correction 
performs very well in all cases, and 
it is similar to the Bonferroni correction in case of sparse signal. 
The arithmetic average e-value and the Simes method with $\ell_K$ correction are not very competitive. 

 Figure~\ref{fig:globalnull_pos}   illustrates that the Simes method with ND correction outperforms the Bonferroni correction substantially in the setting of positive dependence; note that both are valid in this setting. 
 On the other hand, although the rejection probabilities of the U2/U3 and product e-values are quite high in  Figure~\ref{fig:globalnull_pos},   they are not theoretically valid in  the  setting of positive dependence, by noting that $E_iE_j$ is not an e-value if the e-values $E_i$ and $E_j$ are positively correlated. 
 The empirical type-I error of the U2/U3 e-values do not seem to exceed the nominal value $0.05$; this is because $1/20$ is practically a conservative bound when applying to e-values.
    
\subsection{Multiple testing procedures with FDR control}
Next, we compare multiple testing procedures with FDR control. 
The setting is similar to the global null testing experiment, but we focus on negative dependence only.
Set the pairwise correlation to be  $-1/(K-1)$ for  $(X_1,\dots,X_K)$. 
Since FDR control is usually applied in large-scale testing, we consider two specifications of $(K,\pi_0)$: $K=10,000$ or $K=100,000$, and $\pi=80\%$.

We let $\delta$ vary in $[2,4]$ and the e-values in \eqref{eq:compute-e} will be computed with averaging $\delta\in [a,b]=[2,4]$. Fix the FDR upper bound as $0.1$. Each simulation will be repeated 1,000 times and we report their average. 
The procedures that we compare are: 
\begin{enumerate}[label=(\alph*)]
\item the   BH procedure (Section~\ref{sec:BH-proc});
\item the BH procedure with $\ell_K$ correction (Section~\ref{sec:BH-proc}); \item the BH procedure with ND correction (first row of Table~\ref{tab:Su-neg}); 
\item the e-BH procedure (applying the BH procedure to $1/E$; \cite{WR22}). 
\end{enumerate}

 The average numbers of discoveries produced by the four methods are reported in Figure~\ref{fig:BH}.
As expected from its definition, none of the other three methods (b), (c) and (d) is as powerful as the BH procedure (a), but the  BH procedure without correction does not have a   theoretical FDR guarantee under negative dependence.
  Both BH with $\ell_K$ correlation and e-BH are valid under arbitrary dependence, and none of them dominates each other.
  The  BH procedure with ND correction performs better than the other two methods (b) and (d).

\begin{figure}[t]
    \centering    \caption{Multiple testing with FDR control (negative Gaussian dependence)}
    \label{fig:BH}
    \includegraphics[width=0.99\textwidth]{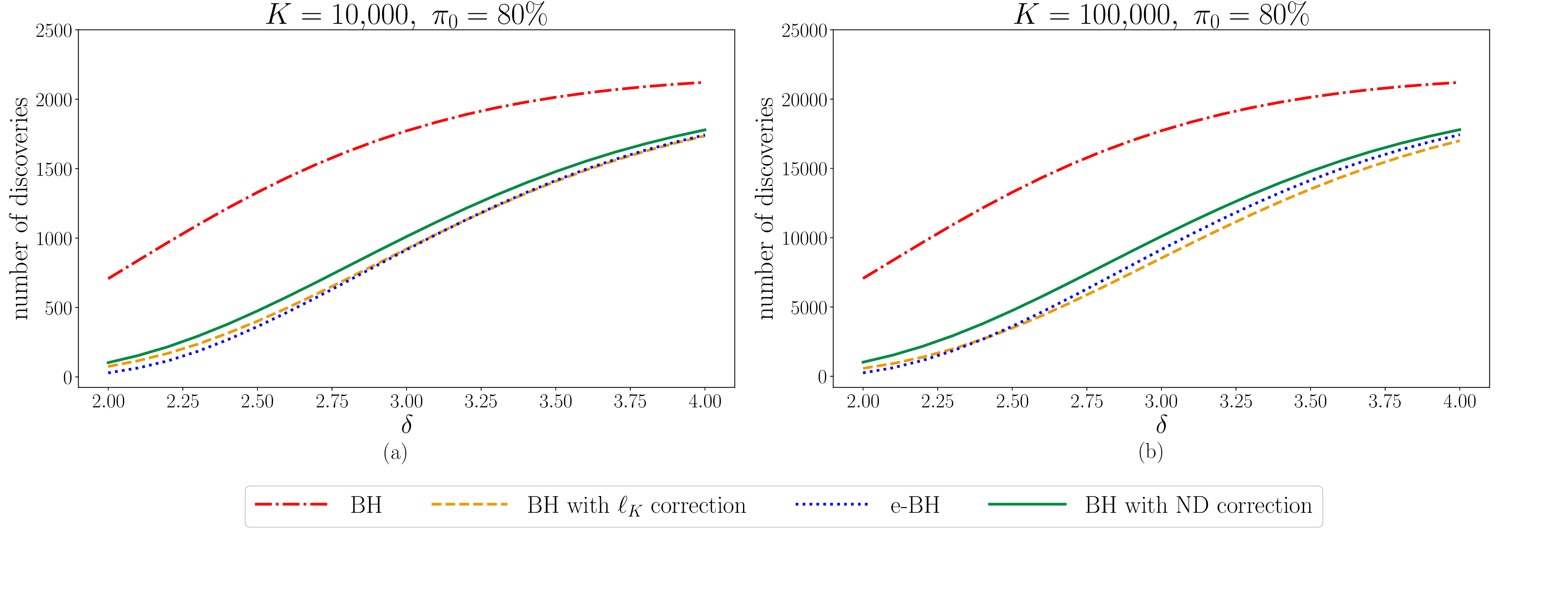}
    \vspace{-1cm}
\end{figure}
  
  Note that the power comparison between the $\ell_K$ correction and the ND correction depends on $K$, because $\ell_K$ explodes as $K$ increases but the ND correction is invariant with respect to $K$. Improvements to the constants derived in our paper can only exaggerate the difference between (c) and (b,d).

\section{Conclusion}\label{sec:conc}

 This paper provides, to our knowledge, the first bounds for multiple testing methods under negative dependence,  in particular the important Simes test and the  BH procedure. Some auxiliary results include error bounds for the weighted Simes test, combining negatively dependent e-values, and some implications under negative Gaussian dependence.

{\textbf{Practical implications.} When faced with negative dependence, a practitioner has two clear alternatives: ignore the dependence and run the BH procedure as is (current practice), or incorporate corrections for the dependence and run BH at a more conservative error level. If practitioners opt for the first choice, then our paper's results show that the inflation in their achieved FDR guarantee will be a simple constant factor, independent of the number of hypotheses. This is in contrast to the BY correction, which suggests an FDR inflation potentially growing logarithmically with the number of hypotheses.} 

{\textbf{Discussions.} 
Seen from \cite{HR95} and Theorem \ref{th:neg}, 
the type-I error of the Simes test under negative dependence  can be (slightly) inflated. An
alternative way to address type-I error inflation is to   explicitly model and estimate the dependence structure and
incorporate it into their decision rule. Copula-based models have been considered in this regard; see e.g., \citet[Section 3.5 and 3.6]{cui2021handbook}. This approach is different from the setting of our paper, as we do not have access to model or estimate the dependence structure explicitly; instead we assume some forms of negative dependence.} 
{Some papers study asymptotic behaviour as $K\to\infty$, which is different from our setting; see  e.g., \cite{delattre2016empirical} and \cite{kluger2024central}.
}

\textbf{Open problems.} 
The most interesting open problem is the BH conjecture for negative dependence, which states that 
    \(\E\left[ {F_{\cD_\alpha}}/{R_{\cD_\alpha}   }\right]\le \frac{K_0}{K} \tilde s_K(\alpha)\)
for any negatively Gaussian dependent vector of p-values, where 
$\tilde s_K(\alpha)$ is in \eqref{eq:sk_neg}.
Alternatively, it will be interesting to obtain other multiplicative corrections of $\alpha$, i.e., without involving the term $-\log(\alpha)$ as in Theorem~\ref{thm:Su-neg}.

Recall that most of our results about Simes and the BH procedure involved the weakest form of negative dependence that we defined (some results about e-values required stronger notions, though).
A second open problem involves the consideration of whether any of the stronger notions of negative dependence (than weak negative dependence) lead to even better bounds for Simes and BH. 
Our Theorem~\ref{thm:Su-neg} assumes  weak negative dependence  among only null p-values, and it remains unclear whether a better bound can be obtained by assuming some form of negative dependence among all p-values.

As the number of hypotheses $K$ increases, p-values may typically be less negatively
correlated on average (the average pairwise correlation must be bounded below by $-1/(K-1)$).
Therefore, one may expect that for certain types of negative dependence, a large $K$ would lead to a better bound than Theorem \ref{th:neg}. However, it is not immediately clear whether this holds true because the large dimension also allows for more flexibility  of dependence (thus, more cases to defend). Resolving this, perhaps for specialized models, is left for future work.

Other open problems include extending our results to adaptive Storey-BH-type procedures, to the weighted BH procedure, and to grouped, hierarchical or multilayer settings~\citep{ramdas2019unified}.  
We hope to make progress on some of these questions in the future.
 
\paragraph{Acknowledgments.} We thank Sanat Sarkar for engaging conversations and stimulating suggestions, as well as a missing reference. We also thank Ajit Tamhane for a relevant reference



\bibliography{extra2}
\bibliographystyle{plainnat}

\appendix
\section{Examples of negative dependence in testing}\label{sec:eg}

We give some examples in  testing and multiple comparisons  where negative dependence naturally appears. 

\subsection{Four examples}
\begin{example}[Tests based on split samples]
\label{ex:1}
Consider a fixed population $\mathbf x=(x_1,\dots,x_m) \in \R^m$,
and suppose that $K$ groups of scientists
are using samples from the population to test their hypotheses. 
The p-value (or e-value) for group $k\in \mathcal K$ 
is computed by $f_k(\mathbf x_{A_k})$, where  $A_k$ 
is a randomly chosen subset of $\{1,\dots,m\}$ and 
$f_k$ is an increasing function.
Since the groups are using different part of the population, 
$A_1,\dots,A_K$ are disjoint sets. 
Using Theorem 2.11 of \cite{JP83}, which says that a permutation distribution is negatively associated, we know that the p-values (or e-values) $f_k(\mathbf x_{A_k})$, $k\in \mathcal K$, as increasing functions of non-overlapping subsets of negatively associated observations (P6 of \cite{JP83}), are negatively associated. 
\end{example}

\begin{example}[Testing the mean of a bag of numbers]  Suppose that the data, represented by the vector $\mathbf X = (X_1,\dots,X_n)$,  are drawn without replacement (uniformly) from a bag of $N$ numbers $x_1,\dots,x_N$, each in $[0,1]$, whose average is $\mu:=\sum_{i=1}^N x_i/N$. Then, we have seen before that $\mathbf X$ is negatively associated. Further, it is clear that $\mathbb{E}[X_i] = \mu$ for each $i$. Thus $(1+\lambda(X_i-\mu))$ is an e-value for any $\lambda \in [-1,1]$, and by Corollary 13 in \cite{chi2022multiple}, $\prod_{i=1}^n (1+\lambda(X_i-\mu))$ is also an e-value. This fact is useful, for example if we want to test $\mu \leq 0.5$ against $\mu>0.5$; in this case, for any $\lambda\in[0,1]$, $\prod_{i=1}^n (1+\lambda(X_i-0.5))$ is an e-value.
\end{example}
\begin{example}[Round-robin tournaments]
\label{ex:2}
 
Imagine that $K$ players play a round-robin tournament (meaning that each pair of players play some number of games against each other). Suppose that we wish to test the 
hypotheses that player $k\in \mathcal K$ has no advantage or disadvantage over any other players, or the
global null hypothesis that all players are equally good. Assume that the game outcomes are independent.
Two players being equally good means that whenever they play a game, both players have equal chance of winning or scoring a certain number. Equivalently, since all sports have player rankings or seedings, the global null hypothesis effectively states that these rankings are irrelevant.
Let $(X_{ij}^m)_{i,j\in \mathcal K,~m=1,\dots,M_{ij}}$ be the results of all games, where
$M_{ij}$ is the number of games played between player $i$ and player $j$,
and $X_{ij}^m= - X_{ji}^m$. The $k$-th hypothesis is that $X_{kj}^m$ is symmetrically distributed around $0$ for all $j$ and $m$.
For $k\in \mathcal K$, let the p-value or e-value be given by $Y_k= f_k(X_{kj}^m: j\ne k,~m=1,\dots,M_{kj})$ for some increasing function $f_k$.
Then $Y_1,\dots,Y_K$ are negatively associated using P6 of \cite{JP83}.   One way to construct p-values and e-values for this testing problem is described in Section \ref{app:A}.
\end{example}

\begin{example}[Cyclical or ordered comparisons]
\label{ex:3}
 
Suppose that $X_1,\dots,X_K$ are independent random variables representing scores of $K$ players in a particular order, e.g., pre-tournament ranking. We are interested in testing whether two players adjacent  in the list have equal skills. The $k$-th null hypothesis, under some assumptions,
$X_k$ and $X_{k+1}$ are identically distributed, where we set $X_{K+1}=X_1$ but we may safely omit $H_K$.  
  For example, for $k\in \mathcal K$, a p-value (or e-value)    may be obtained in the form  $ f_k(X_k-X_{k+1})  $ for some increasing   function $f_k$,
  since the score differences between two adjacent players are useful statistics.  
  Indeed, we can show a stronger result:
For any component-wise  increasing  functions $h_k:\R^2 \to \R$, $i\in \mathcal K$  and independent random variables $X_1,\dots,X_K$, let $Y_k=
 h_k(X_k,-X_{k+1}) $, $k\in \mathcal K$, 
where either $X_{K+1}=X_1$ or $X_{K+1}$ is independent of $(X_1,\dots,X_K)$. Then, the random vector $(Y_1,\dots,Y_K)$ is negative orthant dependent. 
 This result is shown in Proposition~\ref{prop:2} in Section \ref{app:A}.

\end{example}

\subsection{Technical details}
\label{app:A}

Below, we first explain the construction of e-values and p-values for the round-robin tournament test in Example \ref{ex:2}, and then show the statement in Example \ref{ex:3} on negative orthant dependence for ordered comparison.

One way to test the hypotheses in Example \ref{ex:2} is to first construct e-values for each game, combine them to get e-values for each pair of players, and then combine them further to get e-values for each individual player. Finally, to test the global null, one can combine e-values across all players using the U-statistic of order 2 or 3. 

We consider the case where only win, lose and draw are possible outcomes of each game; the case of general scores is similar. Our e-values for a single game are constructed using the principle of testing by betting~\citep{shafer2021testing}. To elaborate, imagine that for the $m$-th game between player $i,j$ we have one (hypothetical) dollar at hand. To form the e-value $E^{(m)}_{ij}$ and we bet some fraction $\epsilon \in [0,1]$ that $i$ will beat $j$. If the game is a draw, our wealth remains 1. If we were right, our wealth increases to $1+\epsilon$, and if we were wrong, it decreases to $1-\epsilon$. $E^{(m)}_{ji}$ is constructed in the opposite fashion: so if $E^{(m)}_{ij}=1+\epsilon$, then $E^{(m)}_{ji}=1-\epsilon$; this is the root cause of the resulting negative dependence. 
Importantly, $\epsilon$ (which could depend on $i,j$, but we omit this for simplicity) must be declared before the game occurs. $E^{(m)}_{ij}$ represents how much we multiplied our wealth due to the $m$-th game and this is an e-value, because under the null hypothesis, there is an equal chance of gaining or losing $\epsilon$, so our expected multiplier equals one.

If a pair of players $i, j$ have played $M_{ij}$ games, let the overall e-value for that pair be defined as $E_{ij} = \prod_{m=1}^{M_{ij}} E^{(m)}_{ij}$. In fact the wealth process across those games forms a nonnegative martingale under the null, since it is the product of independent unit mean terms; however we will not require this martingale property in the current analysis. A large $E_{ij}$ means that player $i$ wins many more games than they lose to $j$. 

Let $E_i$ denotes the e-value for each player $i\in \mathcal K$, that is, $E_i = \prod_{j=1, j \neq i}^K E_{ij}$. Each $E_i$ is an e-value for the same reason as before: it is a product of independent unit mean terms. If $E_i$ is large, it reflects that player $i$ more frequently beat other players than lost to them.

Using Properties P1 and P7 of \cite{JP83}, $(E_{ij})_{i,j\in \mathcal K}$ is negatively associated because its components are constructed from mutually independent random vectors $(E_{ij},E_{ji})$ and each of these vectors is  counter-monotonic (hence negatively associated).  
We can further see that $(E_1,\dots,E_K)$ is 
 also negatively associated, because  each $E_k$ is an  increasing function of $(E_{kj})_{k\in \mathcal K}$ (P6 of \cite{JP83}).
Thus, a final e-value for the global null test can be calculated using the U-statistic of order 2 in Equation 28 of \cite{chi2022multiple}, $E := \sum_{i < j} E_i E_j/{K\choose{2}}$, 
or any other U-statistics as guaranteed by Corollary 13 of \cite{chi2022multiple}.

Next, we show a result verifying the claim of negative orthant dependence in Example \ref{ex:3}. 
 
\begin{proposition}
\label{prop:2}
For any component-wise  increasing  functions $h_i:\R^2 \to \R$, $i\in \mathcal K$  and independent random variables $X_1,\dots,X_K$, let $Y_i=
 h_i(X_i,-X_{i+1}) $, $i\in \mathcal K$, 
where either $X_{K+1}=X_1$ or $X_{K+1}$ is independent of $(X_1,\dots,X_K)$. Then, the random vector $(Y_1,\dots,Y_K)$ is negative orthant dependent. 
\end{proposition}

 \begin{proof}
 
Note that negative upper orthant dependence is equivalent to Equation 6 of \cite{chi2022multiple}, 
 and the analogue holds for negative lower orthant dependence by replacing increasing functions with decreasing ones. Hence, it suffices to show 
\begin{align}\label{eq:ineq}
\E\left[\prod_{i=1}^K  f_i(X_i,-X_{i+1})\right] \le\prod_{i=1}^K \E\left[ f_i(X_i,-X_{i+1})\right] ,
\end{align}
for non-negative  component-wise  increasing functions $f_1,\dots,f_K$, and for non-negative  component-wise decreasing functions $f_1,\dots,f_K$. We only show the first case, as the second is similar.
 
 There is nothing to show if $K=1$; we assume $K\ge 2$ in what follows. First, we consider the case $X_{K+1}=X_1$. 
Let $\mathbf X':=(X'_1,\dots,X_K')$ be an independent copy of $\mathbf X:=(X_1,\dots,X_K)$.  
Define a function $g:\R^{2K} \to \R$ by 
$$
g(x_1,\dots,x_K,x_1',\dots,x_K') = \prod_{i=1}^K  f_i(x_i,-x_{i+1}').
$$
We first claim that for any $(x_2,\dots,x_K,x_2',\dots,x_K')\in \R^{2K-2}$, it holds that
\begin{equation}\label{eq:pair-nd1} \E[g(X_1,x_2,\dots,x_K,X_1,x_2',\dots,x_K') ]
\le \E [g(X_1,x_2,\dots,x_K,X_1',x_2',\dots,x_K') ].
\end{equation}
To see this, it suffices to observe  
\begin{align*}
\E[ f_1(X_1,-x_2)f_2(x_K,-X_1)]\le \E[ f_1(X_1,-x_2)f_2(x_K,-X'_1)]
\end{align*}
due to the Fr\'echet-Hoeffding (or Hardy-Littlewood)  inequality (e.g., \cite[Theorem 3.13]{R13}) because $ f_1(X_1,-x_2)$ and $ f_2(x_K,-X_1) $ are counter-monotonic. 
  Therefore,  \eqref{eq:pair-nd1} holds. 
  It follows that 
  \begin{equation}\label{eq:pair-nd2} \E[g(\mathbf X,X_1,Z_2,\dots,Z_K) ]
\le \E [g(\mathbf X,X_1',Z_2,\dots,Z_K) ].
\end{equation}
holds for all random variables $Z_1,\dots,Z_K$ (here $Z_1$ does not appear).
Using the above argument on $X_2$ we get
 $$ \E[g(\mathbf X,Z_1,X_2,Z_3,\dots,Z_K) ]
\le \E [g(\mathbf X,Z_1,X_2',Z_3,\dots,Z_K) ]
$$
holds for all random variables $Z_1,\dots,Z_K$  (here $Z_2$ does not appear).
Letting $Z_1= X_1'$ we obtain
  \begin{equation}\label{eq:pair-nd4} \E[g(\mathbf X,X_1',X_2,Z_3,\dots,Z_K) ]
\le \E [g(\mathbf X,X_1',X_2',Z_3,\dots,Z_K) ].
\end{equation}
Putting \eqref{eq:pair-nd2} and \eqref{eq:pair-nd4} together we get
 $$ \E[g(\mathbf X,X_1,X_2,Z_3,\dots,Z_K) ]
\le \E [g(\mathbf X,X_1',X_2',Z_3,\dots,Z_K) ].
 $$ 
Repeating the above procedure $K$ times we get 
$$\E[g(\mathbf X,\mathbf X) ]
\le \E [g(\mathbf X,\mathbf X') ],$$
and hence 
\begin{align*} 
\E\left[\prod_{i=1}^K  f_i(X_i,-X_{i+1})\right] 
&= \E[g(\mathbf X,\mathbf X) ] 
\\& \le \E [g(\mathbf X,\mathbf X') ]  =\prod_{i=1}^K \E\left[ f_i(X_i,-X'_{i+1})\right] =\prod_{i=1}^K \E\left[ f_i(X_i,-X_{i+1})\right].
\end{align*} 
Therefore, \eqref{eq:ineq} holds.

If we take $f_K=1$, then \eqref{eq:ineq} becomes $$ 
\E\left[\prod_{i=1}^{K-1}  f_i(X_i,-X_{i+1})\right] \le\prod_{i=1}^{K-1} \E\left[ f_i(X_i,-X_{i+1})\right]  
 $$  
for all independent $X_1,\dots,X_{K}$.
Since $K$ is arbitrary, by moving from $K$ to $K+1$ we obtain that   \eqref{eq:ineq} holds
for all independent $X_1,\dots,X_{K+1}$.   
\end{proof} 
 
\section{Additional Simulation Results}\label{sec:add_sims}

{Some additional simulation results on global null testing described in Section 6.2 of \cite{chi2022multiple} with $K=10$ and $K=1,000$  are  provided in  Figure \ref{fig:enter-label}. The results are qualitatively similar to the case of $K=100$ presented in Section 6.2 of \cite{chi2022multiple}.} 

\begin{figure}[ht]
    \centering
    \caption{Global null testing (negative Gaussian dependence). All six subplots show power against $\delta$. The left endpoint of $\delta=0$ actually represents the achieved type I error, which appears to be controlled at the nominal level $\alpha=0.05$ for all methods. The key observations are in the text.}
    \includegraphics[width=0.99\textwidth]{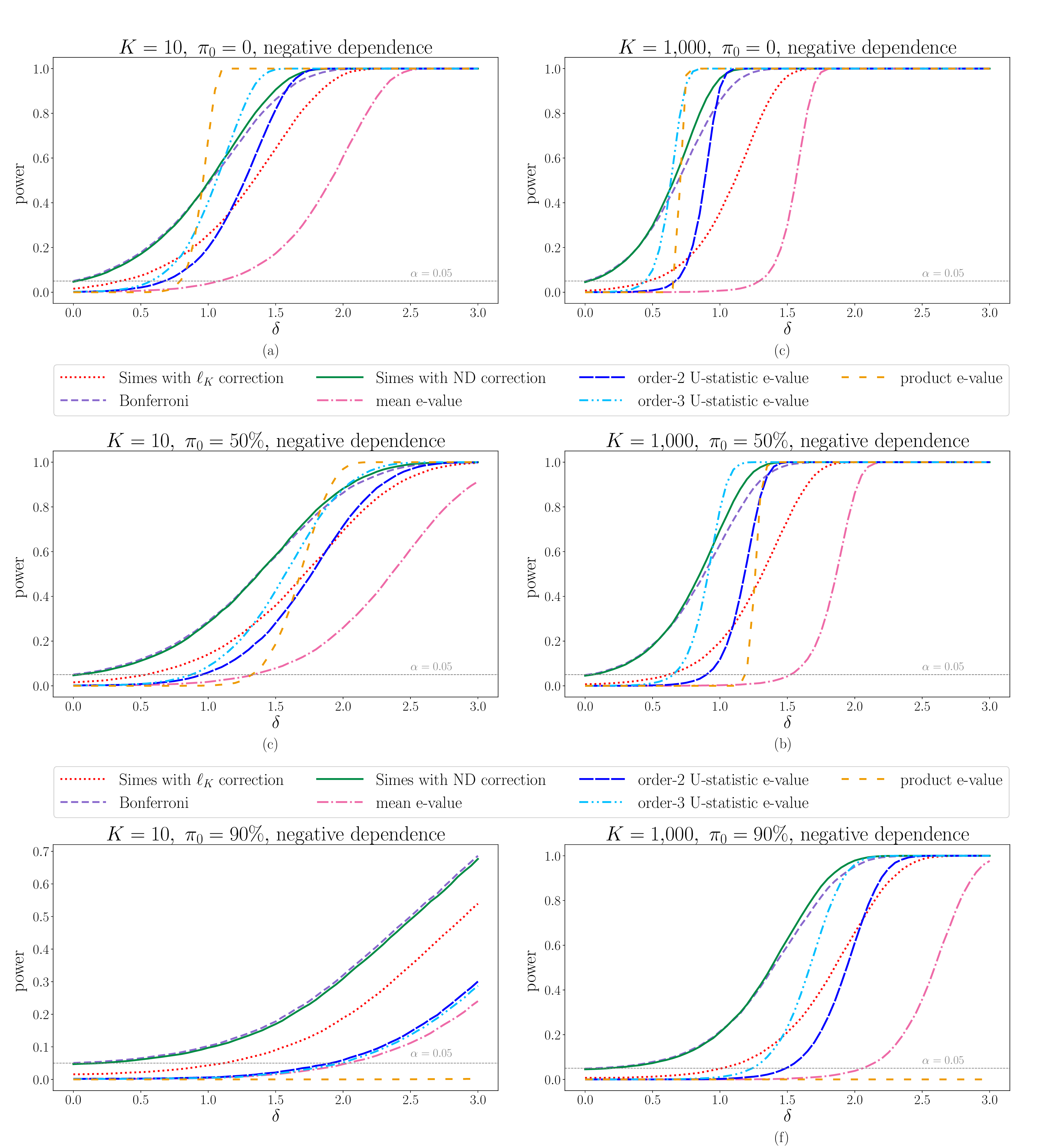}
    \label{fig:enter-label}
\end{figure}

\section{A few technical results used in the proofs}
Below, we report a few technical results used to prove Theorems 6 and 17 of \cite{chi2022multiple}. Define the  harmonic average function $M_{-1,K}:(p_1,\dots,p_K)\mapsto ((p_1^{-1}+\dots+p_K^{-1})/K)^{-1}$.

\begin{theorem}[Theorem 2 (ii) of \cite{CLTW22}]
Suppose that $(P_1,\dots,P_K)\in \mathcal U^K$, and each pair $(P_{i},P_{j})$ follows a bivariate Gaussian copula.
Then $\mathbb{P}\left(M_{-1,K}(P_{1},\dots,P_{K})<\alpha\right) / \alpha \to 1,~\text{as}~\alpha\downarrow 0.$
\end{theorem}

\begin{theorem}[Theorem  3 of \cite{CLTW22}]
 For $(p_1,\dots,p_K)\in [0,1]^K$, $M_{-1,K}(p_{1},\dots,p_{K}) \le S_K(p_1,\dots,p_K)$, and the   inequality holds as an equality if $p_{1}=\dots=p_{K}$.
\end{theorem}

\begin{theorem}[Theorem 1 of \cite{S18}] 
For  arbitrary p-values, the BH procedure $\cD_\alpha$ at level $\alpha$ satisfies   $$  
     \mathrm{FDR}_{\cD_\alpha} \le \alpha + \alpha \int_\alpha^1 \frac{\mathrm{FDR}_0(x)}{x^2} \d x,
 $$  
 where $\mathrm{FDR}_0(x)$ is the FDR of the BH procedure applied to only the null p-values at level $x$.
\end{theorem}

\end{document}